\documentclass{amsart}
\usepackage{graphicx}
\usepackage{amsmath}
\usepackage{amscd}
\usepackage{mathtools}
\usepackage{enumerate}
\usepackage{tikz-cd}
\usepackage[all,cmtip]{xy}
\usepackage{tikz}
\usepackage{amssymb}
\usepackage{ascmac} 
\usepackage{color}
\usepackage{autobreak}
\usepackage{breqn}
\newtheorem{dfn}{Definition}[section]
\newtheorem{thm}[dfn]{Theorem}
\newtheorem{pro}[dfn]{Proposition}
\newtheorem{lem}[dfn]{Lemma}

\theoremstyle{definition}

\newcommand{\co}{\colon}

\newcommand{\mb}{\mathbb}
\newcommand{\ol}{\overline}
\newcommand{\mac}{\mathcal}
\newcommand{\tb}{\textbf}
\newcommand{\ep}{\emph}
\newcommand{\F}{F(X_0,\ldots,X_{n+1})}
\newcommand{\X}{X_1,\ldots,X_{n+1}}
\newcommand{\mca}{\mathcal}
\newcommand{\rmc}{{\rm char}}
\newcommand{\y}{y_1,\ldots,y_n}
\newcommand{\OK}{O_{K_l}}
\newcommand{\K}{K_l}
\newcommand{\kQ}{\kappa(Q)}
\title[Persistence of Galois property across other characteristics]{Persistence of Galois property of hypersurfaces over algebraic integers across other characteristics}
\author{Taro Hayashi}
\author{Kento Otsuka}
\author{Keika Shimahara}
\author{Eito Naruse}
\address{(Taro Hayashi)
Department of Mathematical Sciences,
Ritsumeikan University,
1$-$1$-$1 Nojihigashi, Kusatsu, Shiga, 525$-$8577, Japan}
\email{haya4taro@gmail.com}
\address{(Kento Otsuka)
Graduate School of Mathematical Sciences,
Ritsumeikan University,
1$-$1$-$1 Nojihigashi, Kusatsu, Shiga, 525$-$8577, Japan}
\email{ra0123xe@ed.ritsumei.ac.jp}
\address{(Keika Shimahara)
	Graduate School of Mathematical Sciences,
	Ritsumeikan University,
	1$-$1$-$1 Nojihigashi, Kusatsu, Shiga, 525$-$8577, Japan}
\email{ra0134hp@ed.ritsumei.ac.jp}
\address{(Eito Naruse)
	Graduate School of Mathematical Sciences,
	Ritsumeikan University,
	1$-$1$-$1 Nojihigashi, Kusatsu, Shiga, 525$-$8577, Japan}
\email{rp0140ir@ed.ritsumei.ac.jp}
\date{\today}
\subjclass{Primary 14J70, 12F10; Secondary 14G17}
\keywords{Galois extensions of function fields; hypersurfaces}
\begin{document}
\maketitle
\begin{abstract}
In this paper, we investigate hypersurfaces defined over a ring of algebraic integers, and show that if the projection from a point induces a Galois extension over either a number field or the residue field associated with a prime ideal satisfying certain conditions, then the Galois property persists under reduction modulo the residue field associated with all but finitely many such prime ideals. Furthermore, for quartic hypersurfaces, we provide necessary and sufficient conditions for the Galois group to be given by a projective linear group, depending on the characteristic of the base field.
\end{abstract}
\section{Introduction}
Let \( k \) be a field. We denote the characteristic of \( k \) by \( \mathrm{char}(k) \), and its algebraic closure by \( \overline{k} \).
For an irreducible algebraic variety $Y$ over $k$, let $k(Y)$ be the function field of $Y$.
The automorphism group of $k(Y)$ naturally coincides with the group of birational transformations ${\rm Bir}(Y)$ of $Y$ onto itself ([\ref{bio:rh}, Chapter I. Theorem 4.4]). 

Let $\mb P^m(k)$ be the projective space of dimension $m$ over $k$, and
let $\mac X$ be an irreducible hypersurface of degree $d\geq 4$ in $\mb P^{n+1}(k)$.
Consider a point $P\in \mathbb P^{n+1}(k)$ and the projection $\pi_P\co \mac X\dashrightarrow \mb ^n(k)$ with center $P$. 
Since $\pi_P$ is a dominant rational map,
we have an extension of function fields $\pi_P^{\ast}\co k(\mb P^n(k))\hookrightarrow  k(\mac X)$ with $[k(\mac X):k(\mb P^n(k))]=d-1$ (resp. $d$) if $\mac X$ is smooth at $P$ (resp. $P\not\in \mac X$).
The point $P\in\mathbb P^{n+1}(k)$ is called \ep{\tb{a Galois point}} for $\mac X$ if the extension $k(\mac X)/k(\mb P^n(k))$ is a Galois extension ([\ref{bio:ft14},\ref{bio:h06},\ref{bio:my00},\ref{bio:y01f},\ref{bio:y01g},\ref{bio:y03}]).
If $P\in \mac X$ $($resp. $P\not\in \mac X)$, then we call $P$ \ep{\tb {an inner Galois point}} $($resp. \ep{\tb {an outer Galois point}}$)$. 

Elements of the projective linear group ${\rm PGL}(n+1, k)$ that preserve $\mac X$ induce automorphisms of $\mac X$, and hence they induce automorphisms of $k(\mac X)$.
For a Galois point $P$ for $\mac X$, if the Galois group $G$ of $k(\mac X)/k(\mb P^n(k))$ is induced by a subgroup of ${\rm PGL}(n+1,k)$,
then $P$ is called \ep{\tb{an extendable Galois point}}.
If $\mac X$ is smooth, then all Galois points for $\mac X$ are extendable over $\ol{k}$.
Indeed, for smooth hypersurfaces, any transformation is an automorphism, ex ofcept in the case of smooth quartic surfaces in $\mathbb{P}^3$ ([\ref{bio:acgh},\ref{bio:mm63}]).
Even in this exceptional case, it was shown that all Galois points for smooth quartic surfaces are extendable ([\ref{bio:y01g}]).

Galois points are generalized and studied in the context of non-irreducible curves ([\ref{bio:it24}]).
Moreover, Galois points play significant roles in the study of automorphisms of smooth hypersurfaces ([\ref{bio:f14},\ref{bio:fmt19},\ref{bio:f22},\ref{bio:hmo18},\ref{bio:hmo22},\ref{bio:th21l},\ref{bio:th21d},\ref{bio:kty01},\ref{bio:mo15}]), Galois coverings of projective spaces ([\ref{bio:th23g},\ref{bio:th24a}]), 
the theory of algebraic-geometric codes ([\ref{bio:f22gc}]), and the graph theory ([\ref{bio:fm23}]).
These studies demonstrate that the defining equations for hypersurfaces with Galois points and their associated Galois groups play a crucial role in extending the applications of Galois points to other research areas.

Previous work on Galois points has primarily focused on fixed base fields. In contrast, the interplay between Galois points and variations in the characteristic of the base field remains largely unexplored. In this paper, we study hypersurfaces defined over a ring of algebraic integers, focusing on cases where the projection from a point defines a Galois extension over either a number field or the residue field associated with a prime ideal satisfying certain conditions. Our main goal is to investigate whether the Galois property of such projections is preserved under reduction modulo the residue field associated with all but finitely many such prime ideals. In addition, we explicitly determine the defining equations of hypersurfaces admitting Galois points, as well as the structure of their associated Galois groups. These results provide new insights into how the Galois property behaves under changes in the characteristic of the base field.

Let $R$ be a ring.
For $d\geq 0$, let 
\[R[Y_1,\ldots ,Y_m]_d\] be the free $R$-module of a form of degree $d$ in the variables $Y_1,\ldots, Y_m$.
The form \( F(Y_1, \ldots, Y_m)\in R[Y_1, \ldots, Y_m] \) is said to be \emph{\tb{primitive}} if the ideal in \( R \) generated by all of its coefficients is equal to \( R \) itself.
\begin{dfn}\label{dfn:1.1}
For \( F(Y_1,\ldots, Y_{m}) \in R[Y_1,\ldots, Y_m]_d \), we write
\[
F(Y_1,\ldots, Y_m) = \sum_{i_1+\cdots+i_m = d} a_{i_1,\ldots, i_m} Y_1^{i_1} \cdots Y_m^{i_m}
\]
with \( a_{i_1,\ldots, i_m} \in R \).  
We define
\[
\mathcal{MC}(F(Y_1,\ldots, Y_m)) := \prod_{\substack{i_1+\cdots+i_m = d \\ a_{i_1,\ldots, i_m} \neq 0}} a_{i_1,\ldots, i_m}.
\]
By convention, if \( F(Y_1,\ldots, Y_m) = 0 \), then we set \( \mathcal{MC}(F(Y_1,\ldots, Y_m)) := 1 \).
\end{dfn}
Let 
\( K \) be a number field, and let \( O_{K} \) be the ring of integers of \( K \).
Let \( Q \) be a nonzero prime ideal of \( K \), and set 
\[\kappa(Q) := O_{K}/Q \] to be the residue field. Suppose that \( Q \cap \mathbb{Z} = (p) \) for a prime number \( p \in \mathbb{Z} \). Then \( \kappa(Q) \) is a finite extension of the finite field \( \mathbb{F}_p \).
Let $\F \in O_K[X_0,\ldots, X_{n+1}]_d$ be a form of degree $d$.
Then $\F$ naturally becomes a form of degree $d$ with coefficients in $\kappa(Q)$, via reduction modulo the prime ideal \( Q \subset O_K \).  
Let \[\mac X_0\subset \mb P^{n+1}(K)\] be the hypersurface defined by $\F=0$, and \[\mac X_Q\subset \mb P^{n+1}(\kappa(Q))\] be the hypersurface defined by $\F=0$, now regarded as a polynomial over $\kappa(Q)$.
\begin{dfn}\label{dfn:1.2}
	Let \( F(X_0,\ldots, X_{n+1}) \in O_{K}[X_0,\ldots, X_{n+1}]_d \) be a form of degree \( d \). We write
	\[
	F(X_0,\ldots, X_{n+1}) = \sum_{i=0}^{d} F_i\, X_0^{d-i},
	\]
	where \( F_i \in O_{K}[X_1, \ldots, X_{n+1}]_i \) for each \( i \).
\begin{enumerate}
\item[$(1)$] For each \( i = 2, \ldots, d-1 \), we define the form \( H_{\mathrm{inner};\, i}(F(X_0,\ldots, X_{n+1})) \in O_{K}[X_1,\ldots, X_{n+1}]_{2i - 1} \) of degree \( 2i - 1 \) by
\[
H_{\mathrm{inner};\, i}(F(X_0,\ldots, X_{n+1})) := \sum_{i \geq j \geq 1}
\binom{d - j}{i - j} F_j\, (-F_2)^{i - j} \cdot \left((d - 1)F_1\right)^{j - 1}.
\]
Moreover, we define 
\[h_{\mathrm{inner}}(\F):=\prod_{i=2}^{d-2}\mac{MC}(H_{\mathrm{inner};\, i}(F(X_0,\ldots, X_{n+1})) )\in O_K.\]
\item[$(2)$] For each \( i = 1, \ldots, d-1 \), we define \( H_{\mathrm{outer};\, i}(F) \in O_{K}[X_1,\ldots, X_{n+1}]_i \) by
\[H_{\mathrm{outer};\, i}(\F) := \sum_{i \geq j \geq 0}\binom{d - j}{i - j} F_j\, (-F_1)^{i - j} \cdot \left(dF_0\right)^j.
\]
Moreover, we define 
\[h_{\mathrm{outer}}(\F):=\prod_{i=2}^{d-1}\mac{MC}(H_{\mathrm{outer};\, i}(F(X_0,\ldots, X_{n+1})) )\in O_K.\]
\end{enumerate}
\end{dfn}
\begin{thm}\label{1.4}
Let \( K_l \) be a number field containing a primitive \( l \)-th root \( e_l \) of unity where $l\geq2$.
Let $P:=[1:0:\cdots:0]\in\mb P^{n+1}(\OK)$ be a point, and let \( F(X_0,\ldots, X_{n+1}) \in O_{K_l}[X_0, \ldots, X_{n+1}]_d \) be a primitive form of degree \( d \), and assume that the hypersurface \( \mathcal{X}_0 \subset \mathbb{P}^{n+1}(\K) \) defined by \( \F =0\) is geometrically irreducible.
Let \( W \subset \mathrm{Spec}\, O_{K_l} \) be a nonempty open subset such that the reduction \( \mathcal{X}_Q \) is irreducible for every point \( Q \in W \).
Write
\[
F(X_0, \ldots, X_{n+1}) = \sum_{i = 0}^d F_i X_0^{d - i},
\]
where \( F_i \in O_{K_l}[X_1, \ldots, X_{n+1}]_i \) for \( i = 0, \ldots, d \).
Let \( h_{\mathrm{inner}}(F), h_{\mathrm{outer}}(F) \in O_{K_l} \) be the elements defined in Definition \ref{dfn:1.2}.
Define the closed subsets
\[
Z_{\mathrm{inner}} := V\left((d - 1)\cdot\mathcal{MC}(F_1)\cdot h_{\mathrm{inner}}(\F)\right)\]
and
\[Z_{\mathrm{outer}} := V\left(d \cdot \mathcal{MC}(F_0) \cdot h_{\mathrm{outer}}(\F)\right)
\]
of \( \mathrm{Spec}\, O_{K_l} \).
Define the open subsets
\[
U_{\mathrm{inner}} := W \setminus Z_{\mathrm{inner}} \quad \text{and} \quad U_{\mathrm{outer}} := W \setminus Z_{\mathrm{outer}}.
\]
\begin{enumerate}
\item[$(1)$]If $l=d-1$ and there exists $Q\in U_{\mathrm{inner}}$ such that $\mac X_Q$ is smooth at $P$, and $P$ is an extendable inner Galois point for $\mac X_Q$,
then for any $Q'\in U_{\mathrm{inner}}$,
$\mac X_{Q'}$ is smooth at $P$, and $P$ is an extendable inner Galois point for $\mac{X}_{Q'}$.
In addition, the Galois group of $k(\mathcal{X}_{Q'})/k(\mathbb{P}^n(\kappa(Q')))$ is a cyclic groups of order $d-1$ for any $Q'\in U_{\mathrm{inner}}$.
\item[$(2)$]If $l=d$ and there exists $Q\in U_{\mathrm{outer}}$ such that $P$ is an extendable outer Galois point for $\mac X_Q$,
then for any $Q'\in U_{\mathrm{outer}}$, $P$ is an extendable outer Galois point for $\mac{X}_{Q'}$.
In addition, the Galois group of $k(\mathcal{X}_{Q'})/k(\mathbb{P}^n(\kappa(Q')))$ is a cyclic groups of order $d$ for any $Q'\in U_{\mathrm{outer}}$.
\end{enumerate}
\end{thm}
Theorem \ref{1.4} follows from Theorems \ref{5.2} and \ref{5.3}.

Next, for quartic hypersurfaces, we provide necessary and sufficient conditions for the existence of an extendable Galois point, depending on the characteristic of the base field. 
\begin{thm}\label{1.1}
Let $k$ be an algebraically closed filed, 
let $P:=[1:0:\cdots:0]\in\mb P^{n+1}(k)$ be a point, 
and let $\mac X\subset \mb P^{n+1}(k)$ be an irreducible quartic hypersurface.
We express the defining equation $F(X_0,\ldots, X_{n+1})$ of $\mac X$ as 
\[F(X_0,\ldots, X_{n+1})=\sum_{i=0}^4F_iX_0^{4-i}\]
where $F_i\in k[X_1,\ldots, X_{n+1}]_i$ for $i=0,\ldots,4$.

We assume that $\mac X$ is smooth at $P$.
Then the point $P$ being an extendable Galois point for $\mac X$ is equivalent to the following conditions:
\begin{enumerate}
\item[$(1)$]If $\rmc(k)=3$, then 
\[F_2=0\quad\mathrm{and}\quad F_3=-G^2F_1\]
where $G\in k[\X]$ is a non-zero form $G\in k[\X]_1$.
\item[$(2)$]If $\rmc(k)\neq3$, then 
\[3F_1F_3-F_2^2=0.\]
\end{enumerate}
In both cases,  the Galois group of $k(\mathcal{X})/k(\mathbb{P}^n(k))$ is a cyclic group of order $3$.

We assume that $P\not\in \mac X$. In this case, \( F_0\neq 0 \). For simplicity, we assume that \( F_0= 1 \). 
Then the point $P$ being an extendable Galois point for $\mac X$ is equivalent to the following conditions:
\begin{enumerate}
\item[$(3)$]If $\rmc(k)=2$, then 
\[F_1=0,\]
and for the polynomial $T^3 + F_2T + F_3$ where $T$ is an indeterminate,
there exist forms \( B_i \in k[\X] \) with $B_i\neq0$ for $i=1,2,3$ such that 
\begin{align*}
	\begin{split}
T^3 + F_2T +F_3=(T-B_1)(T-B_2)(T-B_3).
\end{split}
\end{align*}
In this case, the Galois group of $k(\mathcal{X})/k(\mathbb{P}^n(k))$ is the direct product of two cyclic groups of order 2.
\item[$(4)$]If $\rmc(k)\neq2$, then 
\[3F_1^2-8F_2=0\quad\mathrm{and}\quad F_1^3-16F_3=0.\]
In this case,   the Galois group of $k(\mathcal{X})/k(\mathbb{P}^n(k))$ is a cyclic group of order $4$.
\end{enumerate}
\end{thm}
Parts $(2)$ and $(4)$ of Theorem \ref{1.1} is known for quartic normal hypersurfaces in the case where the characteristic of the base field is zero; see, for example, [\ref{bio:ft14}, {\rm Key Lemma}\ 3.1]. Part $(2)$ (resp. $(4)$) of Theorem \ref{1.1} extends this result to the case where the characteristic is not equal to $3$ (resp. $2$).
Theorem \ref{1.1} is followed by Theorems \ref{33}, and \ref{43}.
The distinction in Theorem \ref{1.1} between cases where the characteristic of the base field is $2$ (resp. $3$) and other cases is based on the result that the Galois group differs depending on whether the characteristic $p$ is a prime factor of $d-1$ (resp. $d$) ([\ref{bio:f07},\ref{bio:ft14},\ref{bio:f142}]).

Section 2 is preliminary.
We explain previous research results about Galois points and provide several results that are used in proving Theorems \ref{1.1} and \ref{1.4}.
In Section 3, we show Theorem \ref{1.4}. 
Theorems \ref{5.1} and \ref{5.2} investigate the Galois property for irreducible hypersurfaces of degree $d\geq4$.
However, we do not examine the Galois property in cases where the prime $p$ is a prime factor of either $(d-1)$ or $d$ in these theorems.
This is because in cases where the prime $p$ is a prime factor of either $(d-1)$ or $d$, the explicit form of the defining equations for hypersurfaces with Galois points remains unknown.
In Section 4 (resp. Section 5), we show parts $(1)$ and $(2)$ (resp. $(3)$ and $(4)$) of Theorem \ref{1.1} by utilizing the necessary and sufficient conditions for a degree $3$ (resp. $4$) field extension to be a Galois extension.

\section{Preliminary}
Let $F(X_0,\X)\in k[X_0,\X]$ be a form of degree $d$.
For $A=(a_{ij})\in{\rm GL}(n+2,k)$, 
we define the action of $A$ on $F(X_0,\X)$ as follows:
\[ A^{\ast}F(X_0,\X):=F\Bigl(\sum_{i=1}^{n+2}a_{1i}X_{i-1},\sum_{i=1}^{n+2}a_{2i}X_{i-1},\ldots,\sum_{i=1}^{n+2}a_{n+2\ i}X_{i-1}\Bigr).\]
Let $\mac X\subset \mb P^{n+1}(k)$ be a hypersurface defined by $F(X_0,\X)=0$.
 If there exists $t \in k$ such that
\[A^{\ast}F(X_0,\X)=tF(X_0,\X),\] then $A$ induces an automorphism of $\mac X$, which we denote by $g_A$.
Let $\pi_P\colon \mac X\dashrightarrow \mathbb P^n(k)$ be the projection with center $P\in \mathbb P^{n+1}(k)$.
We assume that $P=[1:0:\cdots:0]$.
Then the map $\pi_P$ is explicitly given by
\[\pi_P\colon \mac X\backslash\{P\}\ni[X_0:X_1:\cdots:X_n:X_{n+1}]\mapsto[X_1:\cdots:X_{n+1}]\in\mathbb P^n(k).\]
Since the automorphism group of $k(\mac X)$ naturally coincides with ${\rm Bir}(\mac X)$,
we get that $P$ is a Galois point for $\mac X$ if and only if $\pi_P\co \mac X\dashrightarrow \mathbb P^n(k)$ is a Galois rational cover.
Define the group
\[
G_{\pi_P} := \{ g \in \mathrm{Bir}(\mathcal{X}) \mid \pi_P \circ g = \pi_P \}
\]
which is a subgroup of $\mathrm{Bir}(\mathcal{X})$. The following results hold:
\begin{enumerate}
\item If $P$ is a Galois point for $\mathcal{X}$, then $G_{\pi_P}$ is the Galois group of $\pi_P$.
\item If $\mathcal{X}$ is smooth at $P$, then the degree of $\pi_P$ is $d-1$. In this case, $P$ is a Galois point for $\mathcal{X}$ if and only if $|G_{\pi_P}| = d-1$.
	\item If $P \notin \mathcal{X}$, then the degree of $\pi_P$ is $d$. In this case, $P$ is a Galois point for $\mathcal{X}$ if and only if $|G_{\pi_P}| = d$.
	\item If $P$ is a Galois point for $\mac X$, then $P$ is an extendable if and only if $G_{\pi_P}$ is generated by automorphisms which are induced by matrices.
\end{enumerate}
We introduce necessary and sufficient conditions for a hypersurface to admit a Galois point when the characteristic of the base field is either $0$ or coprime to $d-1$ or $d$.
Here, $I_m$ is the identity matrix of size $m\geq 1$.
\begin{thm}\label{2.1}$([\ref{bio:ft14},\ref{bio:y01g},\ref{bio:y03}])$.
Let $k$ be an algebraically closed field of $\rmc(k)=0$, 
let $\mac X$ be a normal hypersurface of degree $d\geq 4$ in $\mathbb P^{n+1}(k)$, and let $P\in\mathbb P^{n+1}(k)$ be a point.
Then we have the following:
	\begin{enumerate}
		\item[$(1)$]We assume that $\mac X$ is smooth at $P$.
	The point $P$ is an inner Galois point of $X$ if and only if by replacing the coordinate system if necessary, $P=[1:0:\cdots:0]$ and  $\mac X$ is defined by
		\[ X_1X_0^{d-1}+F_d(\X)=0\]
where $F_d(\X)\in k[\X]_d$.
In addition, the Galois group $G_{\pi_P}$ is generated by an automorphism of $\mac X$ induced by the following matrix:
		\[\begin{pmatrix}
			e_{d-1}&0\\
			0&I_{n+1}
		\end{pmatrix}
		\]
where $e_{d-1}$ is a primitive $d-1$-th root of unity.
\item[$(2)$]We assume that $P\not\in \mac X$.
		The point $P$ is an outer Galois point of $\mac X$ if and only if by replacing the coordinate system if necessary, $P=[1:0:\cdots:0]$ and $\mac X$ is defined by
		\[ X_0^d+F_d(\X)=0\]
where $F_d(\X)\in k[\X]_d$.		
	In addition, the Galois group $G_{\pi_P}$ is generated by an automorphism of $\mac X$ induced by the following matrix:
		\[\begin{pmatrix}
			e_d&0\\
			0&I_{n+1}
\end{pmatrix}\]
where $e_d$ is a primitive $d$-th root of unity.
\end{enumerate}
\end{thm} 
\begin{thm}\label{2.2}$([\ref{bio:f08},\ {\rm Proposition}\ 1\ {\rm and}\ {\rm Proposition}\ 3])$.
	Let $k$ be an algebraically closed field of $\rmc(k)\geq3$,
let $\mac X$ be a smooth curve of degree $d\geq 4$ in $\mathbb P^2(k)$, and let $P\in\mathbb P^2(k)$ be a point.
Then we have the following:
\begin{enumerate}
\item[$(1))$]We assume that $\rmc(k)$ is not a prime factor of $d-1$.
		The point $P$ is an inner Galois point of $\mac X$ if and only if by replacing the coordinate system if necessary, $P=[1:0:0]$ and  $\mac X$ is defined by
		\[ X_1X_0^{d-1}+F_d(X_1,X_2)=0\]
		where $F_d(X_1,X_2)\in k[X_1,X_2]_d$.
		In addition, the Galois group $G_{\pi_P}$ is generated by an automorphism of $\mac X$ induced by the following matrix:
		\[\begin{pmatrix}
			e_{d-1}&0\\
			0&I_2
		\end{pmatrix}
		\]
		where $e_{d-1}$ is a primitive $d-1$-th root of unity.
\item[$(2)$]We assume that $\rmc(k)$ is not a prime factor of $d$.
The point $P$ is an outer Galois point of $\mac X$ if and only if
		by replacing the coordinate system if necessary, $P=[1:0:0]$ and $X$ is defined by
		\[ X_0^d+F_d(X_1,X_2)=0\]
		where $F_d(X_1,X_2)\in k[X_1,X_2]_d$.		
		In addition, the Galois group $G_{\pi_P}$ is generated by an automorphism of $\mac X$ induced by the following matrix:
		\[\begin{pmatrix}
			e_d&0\\
			0&I_2
		\end{pmatrix}
		\]
where $e_d$ is a primitive $d$-th root of unity.
\end{enumerate}
\end{thm} 
When the characteristic of the base field is a prime divisor of \(d-1\) or \(d\), the defining equation of a hypersurface possessing a Galois point has not been determined under general conditions, to the best of the author's knowledge. 
As related results to our main theorem, we introduce two specific cases: the case of inner Galois points for nonsingular plane quartic curves in characteristic \(3\), and the case of extendable outer Galois points when the characteristic \(p\) of the base field satisfies \(d = p^l\) for \(l \geq 1\).
\begin{thm}\label{2.3}$([\ref{bio:f06},\ {\rm Proposition}\ 1])$.
Let $k$ be an algebraically closed field of $\rmc(k)=3$, and let $[X:Y:Z]$ be the coordinate system of $\mb P^2(k)$.
	A smooth quartic curve $C\subset\mb P^2_k$ has an inner Galois point if and only if $C$ is
	projectively equivalent to the curve defined by either of the following two forms:
\begin{enumerate}
\item[$(1)$]$X^{3}Z-XZ^{3}+Y^{4}+a_{3}Y^{3}Z+a_{2}Y^{2}Z^{2}+a_{1}YZ^{3}$,
\item[$(2)$]$X^{3}Y-XYZ^{2}+a_{3}Y^{3}Z+a_{2}Y^{2}Z^{2}+a_{1}YZ^{3}+Z^{4}$,
\end{enumerate}
where $a_{i}(i=1,2,3)$ are constants and in the case $(2),$ $a_{3}\neq 0$ . Furthermore, the
dual map of $C$ is inseparable onto its dual if and only if $C$ is in the case $(1)$ with
$a_{2}=0$.
\end{thm}
\begin{thm}\label{2.4}$([\ref{bio:f142},\ {\rm Proposition}\ 1])$.
Let $k$ be an algebraically closed field of characteristic $p>0$, and let $[X:Y:Z]$ be the coordinate system of $\mb P^2(k)$.
Let $C \subset \mathbb{P}^2(k)$ be an irreducible curve of degree $p^e$. Assume that $P =[1 : 0 : 0]\in\mb P^2(k)$ is an extendable outer Galois point, and let 
	\[
	g_0(x, y) = \prod_{\tilde{\sigma} \in \tilde{G}_P} \tilde{\sigma}^* x = \prod_{\sigma \in G_P} (x + a_{12}(\sigma)y + a_{13}(\sigma)),
	\]
where $x:=\frac{X}{Z}$ and $y:=\frac{Y}{Z}$.

Then, we have the following:
\begin{enumerate}
\item[$(1)$] $G_{\pi_P} \cong (\mathbb{Z}/p\mathbb{Z})^e$.
\item[$(2)$] $g_0(x, y) \in K[y][x]$ has only terms of degree a power of $p$ in the variable $x$.
\item[$(3)$] The defining equation of $C$ on the affine space $Z\neq0$ is of the form $g_0(x, y) + h(y) = 0$.
\end{enumerate}
\end{thm}
These theorems are algebraic in nature, relying on coordinate transformations in projective spaces.  
In $[\ref{bio:th23r}]$, it is showed that the Galois property of a given point $P$ can also be determined geometrically by examining the branching divisors and ramification indices of the projection $\pi_P$.
It should also be noted that when the characteristic of the base field is a prime factor of \( d-1 \) or \( d \), the phenomenon occurs that the Galois group $G_{\pi_P}$ of a Galois point \( P \) is not necessarily a cyclic group, nor even an abelian group.
\begin{thm}\label{2.5}$([\ref{bio:ft14}])$.
Let $k$ be an algebraically closed field of characteristic $p>0$,
and	let $\mac X \subset \mb P^{n+1}(k)$ be a normal hypersurface of degree $d \geq 3$.
	Let $P \in \mb P^{n+1}(k)$ be an inner $($resp. outer$)$ Galois point for $\mac X$, and let $d - 1 = p^e l$ $($resp. $d = p^e l)$, where $l$ is not divisible by $p$.
	Then $l$ divides $p^e - 1$ and 
	\[
	G_{\pi_P}\cong (\mathbb{Z}/p\mathbb{Z})^{\oplus e} \rtimes \mathbb{Z}/l\mathbb{Z}.
	\]
\end{thm}
Theorem \ref{2.2} can be generalized to higher dimensions. 
The proof follows the same basic ideas as in [\ref{bio:f08},\ {\rm Proposition}\ 1\ {\rm and}\ {\rm Proposition}\ 3].
\begin{pro}\label{2.6}
Let $k$ be a field, 
let $P:=[1:0:\cdots:0]\in\mb P^{n+1}(k)$ be a point, 
let $\mac X\subset \mb P^{n+1}(k)$ be a geometrically irreducible hypersurface of degree $d\geq4$, 
and let $F(X_0,\X)$ be the defining equation of $\mac X$.
	
We assume that $\rmc(k)$ is not a prime factor of $d-1$ $($resp. $d)$, and that $k$ has a primitive $d-1$-th $($resp. $d$-th$)$ root $e_{d-1}$ $($resp. $e_d)$ of unity.
The point $P$ is an extendable inner $($resp. outer$)$ Galois point for $\mac X$ if and only if there exists a matrix $A=(a_{ij})\in {\rm GL}(n+2,k)$ such that 
\[a_{i1}=0\quad {\rm for}\ i=2,\ldots, n+2\]
 and 
\[A^{\ast}F(X_0,\X)=X_{n+1}X_0^{d-1}+H\quad({\rm resp}.\ X_0^d+H)\]
where $H\in k[\X]_d$. 
In addition, the Galois group $G_{\pi_P}$ is generated by an automorphism of $\mac X$ induced by the following matrix:
\[\begin{pmatrix}
	e_{d-1}&0\\
	0&I_{n+1}
\end{pmatrix}\quad (\mathrm{resp}.\ \begin{pmatrix}
e_d&0\\
0&I_{n+1}
\end{pmatrix}).\]
\end{pro}
\begin{proof}
We assume that $\rmc(k)$ is not a prime factor of $d-1$, and that $k$ has a primitive $d-1$-th root $e_{d-1}$ of unity.
First, consider the hypersurface defined by the equation
\[F(X_0,\X)=X_{n+1}X_0^{d-1}+H\] where $H\in k[\X]_d$.
We set $S:=\begin{pmatrix}
	e_{d-1}&0\\
	0&I_{n+1}
\end{pmatrix}$.
Then, we have
\[S^{\ast}F(X_0,\X)=F(X_0,\X).\]
	This implies that
	the matrix $S$ defines an automorphism $g_S$ of $\mac X$ with order $d-1$. 
	Since the matrix \( S \) acts trivially on the components $X_1$ through $X_{n+1}$, 
	it follows that 
\[\pi_P \circ g_S = \pi_P.\]
Since  the degree of $\pi_P$ is $d-1$, $G_{\pi_P}$ is generated by $g_S$. 
Therefore, the point $P$ is an extendable inner Galois point for $\mathcal{X}$.
Next, we assume that $\mac X$ is smooth at $P$, and $P$ is an extendable inner Galois point for $\mac X$. 
Let $S:=(s_{ij})\in{\rm GL}(n+2,k)$ be a matrix such that $\pi_P\circ g_S=\pi_P$ where $g_S$ is the automorphism of $\mac X$ induced by $S$.
To justify the following polynomial identity, we temporarily work over the algebraic closure $\bar{k}$ of $k$.
Let $\ol{\mathcal{X}} := \mathcal{X} \times_{{\rm spec}\,k} {\rm Spec}\,\bar{k}$ be the base change of $\mathcal{X}$ to $\ol{k}$.
Since $\mathcal{X}$ is geometrically irreducible, $\ol{\mathcal{X}}$ is irreducible over $\bar{k}$.
Since
\[{\pi_P}_{\mid \ol{\mac X}\backslash \{P\}}\co \ol{\mac X}\backslash\{P\}\ni[X_0:X_1:\cdots :X_{n+1}]\mapsto[X_1:\cdots:X_{n+1}]\in\mb P^n(k), \]
we have
\[[\sum_{i=1}^{n+2}s_{2i}X_{i-1}:\cdots :\sum_{i=1}^{n+2}s_{n+2\,i-1}X_i]=[X_1:\cdots:X_{n+1}]\]
for a general point $[X_0:X_1:\cdots :X_{n+1}]\in  \ol{\mac X}$.
Since $\ol{\mac X}$ is irreducible, it follows that
\begin{equation*}
	\bigl(\sum_{i=1}^{n+2}s_{u+1\,i}X_{i-1}\bigr)X_v-\bigl(\sum_{i=1}^{n+2}s_{v+1\,i}X_{i-1}\bigr)X_u=0 
\end{equation*}
for any point $[X_0:X_1:\cdots :X_{n+1}]\in \ol{\mac X}$ and $u,v=1,\ldots, n+1$.
Since the degree of $\ol{\mac X}$ is at least $4$, we get that 
$\bigl(\sum_{i=1}^{n+2}s_{u+1\,i}X_{i-1}\bigr)X_v-\bigl(\sum_{i=1}^{n+2}s_{v+1\,i}X_{i-1}\bigr)X_u=0$ as polynomials for  $u,v=1,\ldots, n+1$ over $\bar{k}$.
Since both sides are polynomials with coefficients in $k$, and the identity holds in $\ol{k}[X_0,\ldots,X_{n+1}]$, it must also hold in $k[X_0,\ldots,X_{n+1}]$.
Consequently, we obtain $s_{22}=s_{ii}$ for $i=2,\ldots, n+2$ and $s_{ij}=0$ for $2\leq i,j\leq n+2$ and $i\neq j$.
For simplicity, we may assume that $s_{ii}=1$ for $i=2,\ldots, n+2$.
We set 
\[S=\begin{pmatrix} 
	s_{11} & s_{12} & \dots  & s_{1n+2} \\
	0 & 1 & \dots  & 0\\
	\vdots & \vdots & \ddots & \vdots \\
	0 & 0& \dots  & 1
\end{pmatrix}.\]
If $s_{11}=1$, then 
\begin{equation}\label{eq:1}
S^n=\begin{pmatrix} 
	1 & ns_{12} & \dots  & ns_{1n+2} \\
	0 & 1 & \dots  & 0\\
	\vdots & \vdots & \ddots & \vdots \\
	0 & 0& \dots  & 1
\end{pmatrix}.
\end{equation}
Since $\pi_P\circ g_S=\pi_P$, and the degree of $\pi_P$ is $d-1$, 
there exists $m\in\mb N$ such that $S^m=I_{n+2}$, and $m$ divides $d-1$.
Since $\rmc(k)$ is not a prime factor of $d-1$,
this contradicts the equation $(\ref{eq:1})$.
Thus, 
\[s_{11}\neq1.\]
This implies that the Galois group $G_{\pi_P}$ is a cyclic group of order $d-1$.
For simplicity, we may assume that $g_S$ is a generator of $G_{\pi_P}$.
Then $s_{11}$ is a primitive $d-1$-th root of unity.
Let $A=(a_{ij})\in{\rm GL}(n+2,k)$ be a matrix such that 
\begin{equation*}
	a_{ij}=
	\begin{cases}
		\frac{-s_{1j}}{1-s_{11}}&{\rm if}\ i=1\ {\rm and}\ j\geq 2, \\
		1&{\rm if}\ i=j,\\
		0&{\rm otherwise}.
	\end{cases}
\end{equation*}
Then
\[ASA^{-1}= 
\begin{pmatrix}
	s_{11} & 0 \\
	0 & I_{n+1}
\end{pmatrix}.\]
Let $\mca Y\subset\mathbb P^{n+1}(k)$ be the hypersurface of degree $d$ defined by 
\[A^{\ast}F(X_0,\X)=0.\]
Since the point $P$ is fixed under the action of the matrix $A$,
 $P$ is an extendable inner Galois point for $\mac Y$, and $\mac Y$ is smooth at $P$.
We set 
\[\begin{aligned}
A^{\ast}F(X_0,\ldots, X_{n+1})=\sum_{i=0}^dG_iX_0^{d-i}
\end{aligned}\]
where $G_i\in k[X_1,\ldots, X_{n+1}]_i$ for $i=0,\ldots,d$.
Since $P\in\mac Y$, we have $G_0=0$.
Since $\mac Y$ is smooth at $P$, we get $G_1\neq 0$.
Since $G_1$ is a linear form, there exists a matrix $B\in{\rm GL}(n+1,k)$ such that 
\[B^{\ast}G_1=X_{n+1}.\]
We set $C:=\begin{pmatrix}
	1& 0 \\
	0 & B
\end{pmatrix}\in{\rm GL}(n+2,k)$.
Then
\[ C^{\ast}\left(A^*\F\right)=X_{n+1}X_0^{d-1}+\sum_{i=2}^d\left(B^{\ast}G_i\right)X_0^{d-i}.\]
The matrix $CASA^{-1}C^{-1}=\begin{pmatrix}
s_{11} & 0 \\
0 & I_{n+1}
\end{pmatrix}=ASA^{-1}$ induces an automorphism $g_{ASA^{-1}}$ of $\mac Y$.
Since $s_{11}^{d-1}=1$, $s_{11}^i\neq1$ for $i=1,\ldots, d-2$, and there exists $t\in k^{\ast}$ such that 
\[(ASA^{-1})^{\ast}\left(\sum_{i=1}^{d}\left(B^{\ast}G_i\right)X_0^{d-i}\right)=t\sum_{i=1}^{d}\left(B^{\ast}G_i\right)X_0^{d-i},\]
we obtain $B^{\ast}G_i=0$ for $i=2,\ldots,d-1$.
Since $B\in{\rm GL}(n+1,k)$, we have 
\[G_i=0\quad {\rm for}\ i=2,\ldots,d-1.\]

For the case where $\mathrm{char}(k)$ is not a prime factor of $d$ and that $k$ has a primitive $d$-th root $e_d$ of unity,
the sufficient and necessary conditions for $P$ to be an extendable outer Galois point can be established in a similar manner as above. 
An important observation is that when $P$ is an extendable outer Galois point for $\mac X$, its Galois group $G_{\pi_P}$ is generated by an automorphism represented by a matrix of the form 
\[
S = \begin{pmatrix} 
	e_d & s_{12} & \dots  & s_{1\,n+2} \\ 
	0 & 1 & \dots  & 0 \\ 
	\vdots & \vdots & \ddots & \vdots \\ 
	0 & 0 & \dots  & 1
\end{pmatrix}
\]
where $e_d$ is a primitive $d$-th root of unity.
Let $A=(a_{ij})\in{\rm GL}(n+2,k)$ be a matrix such that 
\begin{equation*}
	a_{ij}=
	\begin{cases}
		\frac{-s_{1j}}{1-e_d}&{\rm if}\ i=1\ {\rm and}\ j\geq 2, \\
		1&{\rm if}\ i=j,\\
		0&{\rm otherwise}.
	\end{cases}
\end{equation*}
Then
 \[ASA^{-1} = \begin{pmatrix}
	e_d & 0 \\ 
	0 & I_{n+1}
\end{pmatrix},\]
and $ASA^{-1}$ defines an automorphism of the hypersurface $\mathcal{Y}$ defined by the equation 
\[\begin{aligned}
	A^{\ast}F(X_0,\X)&=\sum_{i=0}^dG_iX_0^{d-i}\\
	& = 0
\end{aligned}\]
where $G_i\in k[\X]_i$ for $i=0,\ldots, d$.
Since the point $P$ is fixed under the action of the matrix $A$,
$P \not\in\mac Y$, and hence $G_0\neq 0$. 
Since $e_d$ is a primitive $d$-th root of unity, it follows that $G_i= 0$ for $i = 1, \ldots, d-1$. 
The remaining details are omitted as they follow directly from the arguments presented above.
\end{proof}
The following proposition is used for the proofs of Theorems \ref{1.4} and \ref{1.1}.
\begin{pro}\label{2.7}
	Let $k$ be a field, 
let $P:=[1:0:\cdots:0]\in\mb P_k^{n+1}$ be a point, 
and let $\mac X\subset \mb P_k^{n+1}$ be a geometry irreducible hypersurface of degree $d\geq4$.
	We express the defining equation $F(X_0,\ldots, X_{n+1})$ of $\mac X$ as 
\[F(X_0,\ldots, X_{n+1})=\sum_{i=0}^dF_iX_0^{d-i}\]
where $F_i\in k[X_1,\ldots, X_{n+1}]_i$ for $i=0,\ldots,d$.
Then we have the following:
\begin{enumerate}
\item[$(1)$]We assume that $\rmc(k)$ is not a prime factor of $d-1$, 
$k$ has a primitive $d-1$-th root $e_{d-1}$ of unity,
$\mac X$ is smooth at $P$,
and $P$ is an extendable inner Galois point for $\mac X$.
Then
	\[F_i=G_{i-1}F_1\]
where $G_{i-1}\in k[\X]_{i-1}$ for $i=2,\ldots, d-1$.
Moreover, if $F_2=0$, then 
\[F_i=0\quad {\rm for}\ i=3,\ldots, d-1.\]
\item[$(2)$]We assume that $\rmc(k)$ is not a prime factor of $d$,
$k$ has a primitive $d$-th root $e_{d}$ of unity,
$P$ is an extendable outer Galois point for $X$.
If $F_1=0$, then 
\[F_i=0\quad {\rm for}\ i=2,\ldots, d-1.\]
\end{enumerate}
\end{pro}
\begin{proof}
We show part $(1)$ of Proposition \ref{2.7}.
Since $\mac X$ is smooth at $P$, we get $F_0=0$ and $F_1\neq 0$.
	By Proposition \ref{2.6}, there exists a matrix $A=(a_{ij})\in{\rm GL}(n+2,k)$ such that $a_{i1}=0$ for $i=2,\ldots, n+2$ and 
	\begin{equation}\label{eq01}
		A^{\ast}F(X_0,\X)=X_{n+1}X_0^{d-1}+G
	\end{equation} where $G\in k[\X]_d$.
	Let $B=(b_{ij})\in{\rm GL}(n+1,k)$ be the matrix defined by $b_{ij}=a_{i-1\,j-1}$ for $2\leq i,j\leq n+2$.
Since $a_{i1}=0$ for $i=2,\ldots, n+2$,
we have
	\begin{dmath*}
		A^{\ast}F(X_0,\X)=\sum_{i=1}^dB^{\ast}F_i\left(\sum_{i=1}^{n+2}a_{1i}X_{i-1}\right)^{d-i}\\
		=\sum_{i=1}^dH_iX_0^{d-i}.
	\end{dmath*}
	where 
\begin{equation}\label{eq011}
H_i=\sum_{j=1}^iB^{\ast}F_j(a_{11})^{d-i}\binom{d-j}{d-i}\left(\sum_{i=2}^{n+2}a_{1i}X_{i-1}\right)^{i-j}
\end{equation}
	for $i=1,\ldots, d$.
By the equation $(\ref{eq01})$, we get 
\[H_i=0\quad {\rm for}\ i=2,\ldots, d-1.\]
Since $A\in {\rm GL}(n+2,k)$, we have $a_{11}\neq0$. 
By the case $i = 2$, we get
\[B^{\ast}F_2=(d-1)B^{\ast}F_1\left(\sum_{i=2}^{n+2}a_{11}X_{i-1}\right).\]
Define
\[G_1:={B^{-1}}^{\ast}\left((d-1)\sum_{i=2}^{n+2}a_{11}X_{i-1}\right)\in k[\X]_1.\]
Then
\[F_2=G_1F_1.\]
	By the equations $(\ref{eq01})$ and using induction, we deduce that 
	\[F_i=G_{i-1}F_1\]
	where $G_{i-1}\in k[\X]_{i-1}$
	for $i=3,\ldots, d-1$.
In particular, if $F_2=0$, then 
$B^{\ast}F_1(d-1)\left(\sum_{i=2}^{n+2}a_{1i}X_{i-1}\right)=0$.
Since $F_1\neq0$, $B\in{\rm GL}(n+1,k)$, and $p$ is coprime to $d-1$, 
	we have 
	\[a_{1i}=0\quad {\rm for}\ i=2,\ldots, n+2.\]
	By the equations $(\ref{eq011})$ and $H_i=0$, we get 
	\[B^{\ast}F_i=0\]
	for $i=3,\ldots, d-1$.
	Since $B\in{\rm GL}(n+1,k)$, it follows that $F_i=0$ for $i=3,\ldots, d-1$.
	By Proposition \ref{2.6}, it follows that $P$ is an extendable inner Galois point for $\mac X$.
	
Next, we show part $(2)$ of Proposition \ref{2.7}.
 This case is proven in the same way as above.
Since $P\not\in \mac X$, we get $F_0\neq 0$.
For simplicity, we may assume that $F_0=1$.
	By Proposition \ref{2.6}, there exists a matrix $A=(a_{ij})\in{\rm GL}(n+2,k)$ such that $a_{i1}=0$ for $i=2,\ldots, n+2$ and 
	\begin{equation}\label{eq02}
		A^{\ast}F(X_0,\X)=X_0^d+G
	\end{equation} where $G\in k[\X]_d$.
	Let $B=(b_{ij})\in{\rm GL}(n+1,k)$ be the matrix defined by $b_{ij}=a_{i+1\,j+1}$ for $1\leq i,j\leq n+1$.
	Then 
	\begin{dmath*}
		A^{\ast}F(X_0,\X)=\sum_{i=0}^dB^{\ast}F_i\left(\sum_{i=0}^{n+2}a_{1i}X_{i-1}\right)^{d-i}\\
		=\sum_{i=0}^dH_iX_0^{d-i}.
	\end{dmath*}
where 
\begin{equation}\label{eq021}
H_i=\sum_{j=0}^iB^{\ast}F_j(a_{11})^{d-i}\binom{d-j}{d-i}\left(\sum_{i=2}^{n+2}a_{1i}X_{i-1}\right)^{i-j}
\end{equation}
	for $i=0,\ldots, d$.
	By the equation $(\ref{eq02})$, we get $H_i=0$ for $i=1,\ldots, d-1$. 
Since $a_{11}\neq0$,
	for the case $i=1$, we have
	\[B^{\ast}F_0d\left(\sum_{i=2}^{n+2}a_{11}X_{i-1}\right)+B^{\ast}F_1=0.\]
If $F_1=0$, then it follows that
	\[a_{1i}=0\quad {\rm for}\ i=2,\ldots, n+2.\]
	By the equations $(\ref{eq021})$ and $H_i=0$, we get 
	\[F_i=0\]
	for $i=2,\ldots, d-1$.
	By Proposition \ref{2.6}, it follows that $P$ is an extendable outer Galois point for $\mac X$.
\end{proof}

\section{Proof of Theorem \ref{1.4}}
Here, we recall the definitions of several symbols and notions introduced earlier.
Let 
\( K \) be a number field, and let \( O_{K} \) be the ring of integers of \( K \).
Let \( Q \) be a nonzero prime ideal of \( K \), and set 
\[\kappa(Q) := O_{K}/Q \] to be the residue field which is a finite extension of the finite field \( \mathbb{F}_p \) where
$p\in\mathbb Z$ is a prime umber such that \( Q \cap \mathbb{Z} = (p) \).
Let $\F \in O_K[X_0,\ldots, X_{n+1}]_d$ be a form of degree $d$.
Then $\F$ naturally becomes a form of degree $d$ with coefficients in $\kappa(Q)$, via reduction modulo the prime ideal \( Q \subset O_K \).
Let \[\mac X_0\subset \mb P^{n+1}(K)\] be the hypersurface defined by $\F=0$, and \[\mac X_Q\subset \mb P^{n+1}(\kappa(Q))\] be the hypersurface defined by $\F=0$, now regarded as a polynomial over $\kappa(Q)$.

For \( F(X_0,\ldots, X_{n+1}) \in O_K[X_0,\ldots, X_{n+1}]_d \), 
let \(\mathcal{MC}(F(X_0,\ldots, X_{n+1}))\) be the product of all coefficients of \(F(X_0,\ldots, X_{n+1})\), defined to be \(1\) if \(F(X_0,\ldots, X_{n+1}) = 0\).

For forms $G(X_0,\ldots, X_{n+1}),H(X_0,\ldots, X_{n+1})\in O_K[X_0,\ldots, X_{n+1}]_d$,
\[
G(X_0,\ldots, X_{n+1}) \equiv_Q H(X_0,\ldots, X_{n+1})
\]
if their images in \( \kappa(Q)[X_0,\ldots, X_{n+1}]_d \) coincide as polynomials over the residue field \( \kappa(Q) \) where $Q$ is a nonzero prime ideal of $O_K$.
\begin{thm}\label{5.1}
Let 
\( K_l \) be a number field containing a primitive \( l \)-th root \( e_l \) of unity where $l\geq2$,
let $P:=[1:0:\cdots:0]\in\mb P^{n+1}(K_l)$ be a point, and
let
\[F(X_0,\ldots, X_{n+1})=\sum_{i=0}^dF_iX_0^{d-i}\in \OK[X_0,\ldots, X_{n+1}]_d\]
be a primitive irreducible form of degree $d\geq4$ where $F_i\in \OK[X_1,\ldots, X_{n+1}]_i$ for $i=0,\ldots,d$.
We assume that the hypersurface \( \mathcal{X}_0 \subset \mathbb{P}^{n+1}(\K) \) defined by \( \F =0\) is geometrically irreducible.
Let \( W \subset \mathrm{Spec}\, \OK \) be a nonempty open subset such that the reduction \( \mathcal{X}_Q \) is irreducible for every point \( Q \in W \).
Then we have the following:
\begin{enumerate}
\item[$(1)$]We assume that $l=d-1$. For $Q\in W$ such that $\mac{MC}(F_1)\not\in Q$ and $\mac{MC}(F_i)\in Q$ for $i=0,2,3,\ldots, d-1$,
 the point $P$ is an extendable inner Galois point for $\mac{X}_Q$ whose Galois group is a cyclic group of order $d-1$. 
 \item[$(2)$]We assume that $l=d$. For $Q\in W$ such that $\mac{MC}(F_0)\not\in Q$ and $\mac{MC}(F_i)\in Q$ for $i=1,\ldots, d-1$,
 the point $P$ is an extendable outer Galois point for $\mac{X}_Q$ whose Galois group is a cyclic group of order $d$. 
\end{enumerate}
\end{thm}
\begin{proof}
	We show part $(1)$ of Theorem \ref{5.1}.
	By the assumption, 
	\[F(X_0,\ldots,X_{n+1})\equiv_Q F_1X_0^{d-1}+F_d.\]
Since $F_1\neq0$ as a polynomial over $\kQ$, $\mac X_Q\subset \mathbb P^{n+1}(\kQ)$ is smooth at $P$.
	By Proposition \ref{2.6}, it follows that $P$ is an inner Galois point for $\mac{X}_Q$ whose Galois group is a cyclic group of order $d-1$. 
	In the same way, we can show part $(2)$ of Theorem \ref{5.1}
\end{proof}
For \( F(X_0,\ldots, X_{n+1}) \in O_K[X_0,\ldots, X_{n+1}]_d \), we write
\[
F(X_0,\ldots, X_{n+1}) = \sum_{i=0}^{d} F_i\, X_0^{d-i},
\]
where \( F_i \in O_{K}[X_1, \ldots, X_{n+1}]_i \) for each \( i \).
 For each \( i = 2, \ldots, d-1 \), we define the form \( H_{\mathrm{inner};\, i}(F(X_0,\ldots, X_{n+1})) \in O_{K}[X_1,\ldots, X_{n+1}]_{2i - 1} \) of degree \( 2i - 1 \) by
	\[
	H_{\mathrm{inner};\, i}(F(X_0,\ldots, X_{n+1})) := \sum_{i \geq j \geq 1}
	\binom{d - j}{i - j} F_j\, (-F_2)^{i - j} \cdot \left((d - 1)F_1\right)^{j - 1}.
	\]
Let
	\[h_{\mathrm{inner}}(\F):=\prod_{i=2}^{d-2}\mac{MC}(H_{\mathrm{inner};\, i}(F(X_0,\ldots, X_{n+1})) )\in O_K,\]
and let
	\[
	Z_{\mathrm{inner}}(\F) := V\left((d - 1)\cdot\mathcal{MC}(F_1)\cdot h_{\mathrm{inner}}(\F)\right)\]
	be a closed subset of \( \mathrm{Spec}\, O_{K} \).
\begin{thm}\label{5.2}
Let \( K \) be a number field containing a primitive \( d-1 \)-th root \( e_{d-1} \) of unity,	
	let $P:=[1:0:\cdots:0]\in\mb P^{n+1}(K)$ be a point, and let \( F(X_0,\ldots, X_{n+1}) \in O_{K}[X_0, \ldots, X_{n+1}]_d \) be a primitive form of degree \( d \).
We assume that the hypersurface \( \mathcal{X}_0 \subset \mathbb{P}^{n+1}(K) \) defined by \( \F =0\) is geometrically irreducible.
Let \( W \subset \mathrm{Spec}\, O_{K} \) be a nonempty open subset such that the reduction \( \mathcal{X}_Q \) is irreducible for every point \( Q \in W \), and let 
\[
U:= W \setminus Z_{\mathrm{inner}}(\F)
\]
be an open subset
where $Z_{\mathrm{inner}}(\F)$ is the closed subset defined above.

If there exists $Q\in U$ such that $\mac X_Q$ is smooth at $P$, and $P$ is an extendable inner Galois point for $\mac X_Q$,
	then for any $Q'\in U$,
	$\mac X_{Q'}$ is smooth at $P$, and $P$ is an extendable inner Galois point for $\mac{X}_{Q'}$.
	In addition, the Galois group of $k(\mathcal{X}_{Q'})/k(\mathbb{P}^n(\kappa(Q')))$ is a cyclic groups of order $d-1$ for any $Q'\in U$.
\end{thm}
\begin{proof}
We write
\[
F(X_0, \ldots, X_{n+1}) = \sum_{i = 0}^d F_i X_0^{d - i},
\]
where \( F_i \in O_{K}[X_1, \ldots, X_{n+1}]_i \) for \( i = 0, \ldots, d \).

First, we assume that $\mac X_0$ is smooth at $P$, and $P$ is an extendable inner Galois point for $\mac X_0$.
Since $\mac X_0$ is smooth at $P$, 
we obtain $F_0=0$ and $F_1\not=0$.
If $F_2=0$, then by Proposition \ref{2.7}, it follows that $F_i=0$ for $i=2,\ldots, d-1$.
By Proposition \ref{2.6}, $P$ is an extendable inner Galois point for $\mac X_Q$ where $Q\in U$. 
We assume that $F_2\neq 0$.
By Proposition \ref{2.7},
	there exist forms $G_{i-1}\in K[\X]_{i-1}$ such that 
	\[F_i=F_1G_{i-1}\]
	for $i=2,\ldots,d-1$.
We set $f_1:=\mac{MC}(F_1)$.
Since $F_2=F_1G_1$, the degree of $F_2$ is $2$, and the degrees of $F_1$ and $G_1$ are $1$, we get
there are $a_1,\ldots,a_{n+1}\in O_K$ such that 
\[G_1=\sum_{i=1}^{n+1}\frac{a_i}{f_1}X_i.\]
Let $A:=(a_{ij})\in {\rm GL}(n+2,K)$ be the matrix defined by
\begin{equation*}
		a_{ij}=
		\begin{cases}
			-\frac{a_{j-1}}{(d-1)f_1}&{\rm if}\ i=1\ {\rm and}\ j\geq 2, \\
			1&{\rm if}\ i=j,\\
			0&{\rm otherwise}.
		\end{cases}
	\end{equation*}
Then, we have 
	\[A^{\ast}F(X_0,\ldots, X_{n+1})=F_1X_0^{d-1}+\sum_{i=3}^{d}H_iX_0^{d-i}\]
	where 
\[H_i=\sum_{i\geq j\geq 1}\binom{d-j}{i-j}F_j\left(-\frac{G_1}{d-1}\right)^{i-j}\in K[\X]_i\]
	for $i=2,\ldots, d$.
	By Proposition \ref{2.7},
\begin{dmath}\label{t1}
H_i=0
\end{dmath}
for $i=2,\ldots, d-1$.

Let $Q\in U$ be a nonzero primes ideal.
Since $f_1\not\in Q$, $F_1\not\equiv_Q 0$.
It follows that $\mac X_Q$ is smooth at $P$.
Since $d-1\not\in Q$, $\left((d-1)f_1\right)^{-1}\in\kappa(Q)$.
Thus, $A\in {\rm GL}(n+2,\kappa(Q))$. 
Since 
\[A^{\ast}F(X_0,\ldots, X_{n+1})\equiv_QF_1X_0^{d-1}-H_d\]
where $F_1\in \kappa(Q)[\X]$ and $H_d\in \kappa(Q)[\X]$,
by Proposition \ref{2.6}, we get that $P$ is an extendable inner Galois point for $\mac{X}_Q$ and its Galois group is a cyclic group of order $d-1$
	
Next, we assume that $\mac X_Q$ is smooth at $P$, and $P$ is an extendable inner Galois point for $\mac X_Q$ where $Q\in U$ is a nonzero prime ideal.
It can be proved using the same approach as the case where $\mac X_0$.
Since $\mac X_Q$ is smooth at $P$, we get that $F_0\equiv_Q0$ and $F_1\not\equiv_Q0$.
Then $F_1\neq0$ as a polynomial with $O_K$ coefficients.
Since $Q\in U$, $\mac{MC}(F_0)\not\in Q$.
Thus, $F_0=0$ as a polynomial with $O_K$ coefficients.
If $F_2\equiv_Q0$,
then by Proposition \ref{2.7}, it follows that 
\[F_i\equiv_Q0\]
for $i=2,\ldots, d-1$.
Since $\mac{MC}(F_i)\not\in Q$, we have 
\[F_i=0\] as a polynomial with $O_K$ coefficients
for $i=2,\ldots, d-1$.
By Proposition \ref{2.6}, $P$ is an extendable inner Galois point for $\mac X_{Q'}$ and its Galois group is a cyclic group of order $d-1$ where $Q'\in U$. 
We assume that $F_2\not\equiv_Q 0$.
	By Proposition \ref{2.6},
	there exists form $G'_1\in \kappa(Q)[\X]_{1}$ such that 
\begin{equation}\label{t2}
F_2\equiv_Q F_1G'_1.
\end{equation}
	We set 
\[G'_1\equiv _Q\sum_{i=1}^{n+1}a_iX_i\]
 where $a_1,\ldots,a_{n+1}\in \kappa(Q)$.
	Let $A:=(a_{ij})\in {\rm GL}(n+2,\kappa(Q))$ be the matrix defined by
	\begin{equation*}
		a_{ij}=
		\begin{cases}
			-\frac{a_{j-1}}{d-1}&{\rm if}\ i=1\ {\rm and}\ j\geq 2, \\
			1&{\rm if}\ i=j,\\
			0&{\rm otherwise}.
		\end{cases}
\end{equation*}
Then, we have 
\[A^{\ast}F(X_0,\ldots, X_{n+1})\equiv_Q F_1X_0^{d-1}+H_d\]
where 
\begin{dmath*}
H_d\equiv_Q\sum_{d\geq j\geq1}\binom{d-j}{d-j}F_j\left(-\frac{G'_1}{d-1}\right)^{d-j}
\end{dmath*}
and 
\begin{dmath}\label{eq6}
0\equiv_Q\sum_{i\geq j\geq1}\binom{d-j}{i-j}F_j\left(-\frac{G'_1}{d-1}\right)^{i-j}
\end{dmath}
for $i=2,\ldots, d-1$.
By multiplying  this equation by $\left((d-1)F_1)\right)^{i-1}$
	for each $i=2,\ldots,d-1$, we have
\begin{align*}
\begin{split}
0\equiv_Q\sum_{i\geq j\geq1}\binom{d-j}{i-j}F_j(-F_2)^{i-j}\left((d-1)F_1\right)^{j-1},
\end{split}
\end{align*}
i.e.
\[	H_{\mathrm{inner};\, i}(F(X_0,\ldots, X_{n+1}))\equiv_Q0\]
for each $i=2,\ldots,d-1$.
Since $Q\in U$, $h_{\mathrm{inner};\, i}(F(X_0,\ldots, X_{n+1}))\not\in Q$ for each $i=2,\ldots,d-1$.
Therefore,
\begin{dmath}\label{eq7}
0=\sum_{i\geq j\geq1}\binom{d-j}{i-j}F_j(-F_2)^{i-j}\left((d-1)F_1\right)^{j-1}
\end{dmath}
as a polynomial with $O_K$ coefficients for each $i=2,\ldots,d-1$.
In particular, for the case $i=3$, we have
	\begin{dmath*}
d(d-3)(F_2)^2=2(d-1)F_1F_3.
	\end{dmath*}
Consequently, $F_1$ is a prime factor of $F_2$.
Since the degree of $F_2$ is $2$, and the degree of $F_1$ is $1$, we get
there exists a form $G_1=\sum_{i=1}^{n+1}a_iX_i\in O_K[\X]_1$
such that
\[F_2=F_1\frac{G_1}{f_1}.\]
	By the equation $(\ref{t2})$, we have
	\[G'_1\equiv_Q\frac{G_1}{f_1}.\]
Thus, the equations in $(\ref{eq6})$ hold for $\frac{G_1}{f_1}$.
By the equations in $(\ref{eq7})$ and Proposition \ref{2.6}, 
we get that for $Q'\in U$, $P$ is an extendable inner Galois point for $\mac{X}_{Q'}$ and its Galois group is a cyclic group of order $d-1$.
\end{proof}
For \( F(X_0,\ldots, X_{n+1}) \in O_K[X_0,\ldots, X_{n+1}]_d \), we write
\[
F(X_0,\ldots, X_{n+1}) = \sum_{i=0}^{d} F_i\, X_0^{d-i},
\]
where \( F_i \in O_{K}[X_1, \ldots, X_{n+1}]_i \) for each \( i \).
For each \( i = 1, \ldots, d-1 \), we define \( H_{\mathrm{outer};\, i}(F) \in O_{K}[X_1,\ldots, X_{n+1}]_i \) by
	\[
	H_{\mathrm{outer};\, i}(\F) := \sum_{i \geq j \geq 0}
	\binom{d - j}{i - j} F_j\, (-F_1)^{i - j} \cdot \left(dF_0\right)^j.
	\]
Let
\[
h_{\mathrm{outer}}(\F):=\prod_{i=2}^{d-1}\mac{MC}(H_{\mathrm{outer};\, i}(F(X_0,\ldots, X_{n+1})) )\in O_K,
\]
and
\[
Z_{\mathrm{outer}}(\F) := V\left(d \cdot \mathcal{MC}(F_0) \cdot h_{\mathrm{outer}}(\F)\right)
\]
be a closed subset of \( \mathrm{Spec}\, O_{K} \).
\begin{thm}\label{5.3}
	Let \( K \) be a number field containing a primitive \( d \)-th root \( e_d \) of unity, let
$P:=[1:0:\cdots:0]\in\mb P^{n+1}(K)$ be a point, and let \( F(X_0,\ldots, X_{n+1}) \in O_{K}[X_0, \ldots, X_{n+1}]_d \) be a primitive form of degree \( d \).
We assume that the hypersurface \( \mathcal{X}_0 \subset \mathbb{P}^{n+1}(K) \) defined by \( \F =0\) is geometrically irreducible.
Let \( W \subset \mathrm{Spec}\, O_{K} \) be a nonempty open subset such that the reduction \( \mathcal{X}_Q \) is irreducible for every point \( Q \in W \), and let
\[
U:= W \setminus Z_{\mathrm{outer}}(\F)
\]
be an open subset where $Z_{\mathrm{outer}}(\F)$ is the closed subset defined above.

If there exists $Q\in U$ such that $\mac X_Q$ is smooth at $P$, and $P$ is an extendable outer Galois point for $\mac X_Q$,
	then for any $Q'\in U$, $P$ is an extendable outer Galois point for $\mac{X}_{Q'}$.
	In addition, the Galois group of $k(\mathcal{X}_{Q'})/k(\mathbb{P}^n(\kappa(Q')))$ is a cyclic groups of order $d$ for any $Q'\in U$.
\end{thm}
\begin{proof}
We write
\[
F(X_0, \ldots, X_{n+1}) = \sum_{i = 0}^d F_i X_0^{d - i},
\]
where \( F_i \in O_{K}[X_1, \ldots, X_{n+1}]_i \) for \( i = 0, \ldots, d \).	
	
First, we assume that $P$ is an extendable outer Galois point for $\mac X_0$.
Since $P\not\in \mac X_0$, $F_0\neq0$. 
	We set $f_0:=\mac{MC}(F_0)$ and 
\[F_1=\sum_{i=1}^{n+1}a_iX_i.\]
Let $A:=(a_{ij})\in {\rm GL}(n+2,K)$ be the matrix defined by
	\begin{equation*}
		a_{ij}=
		\begin{cases}
			-\frac{a_{j-1}}{df_0}&{\rm if}\ i=1\ {\rm and}\ j\geq 2, \\
			1&{\rm if}\ i=j,\\
			0&{\rm otherwise}.
		\end{cases}
	\end{equation*}
	Then, we have 
\begin{equation}\label{eqq1}
A^*\F=F_0X_0^d+\sum_{i=2}^dH_iX_0^{d-i}
\end{equation}
	 where
	\begin{dmath*}
		H_i=\sum_{i\geq j\geq0}\binom{d-j}{i-j}F_j\left(-\frac{F_1}{df_0}\right)^{i-j}
	\end{dmath*}
	for $i=2,\ldots, d$.
	By Proposition \ref{2.7},
\begin{dmath}\label{t4}
0=H_i
\end{dmath}
for $i=2,\ldots, d-1$.
Let $Q\in U$.
Since $df_0\not\in Q$,
we get that $F_0\not\equiv_Q0$ and $A\in {\rm GL}(n+2,\kappa(Q))$.
By the equations in $(\ref{eqq1})$ and $(\ref{t4})$,
 \[A^{\ast}F(X_0,\X)\equiv_QF_0X_0^d-H_d.\]
By Proposition \ref{2.6}, $P$ is an extendable outer Galois point for $\mac{X}_Q$ and its Galois group is a cyclic group of order $d$.
	
Next, we assume that $P$ is an extendable outer Galois point for $\mac X_Q$ where $Q\in U$ is a nonzero prime ideal.
	It can be proved using the same approach as the case where $\mac X_0$.
By using $F_1=\sum_{i=1}^{n+1}a_iX_i$,
we define the matrix $A$ as above. 
Then 
\[A^*\F\equiv_QF_0X_0^d+\sum_{i=2}^dH_iX_0^{d-i}\]
where
\[
	H_i:=\sum_{i\geq j\geq0}\binom{d-j}{i-j}F_j\left(-\frac{F_1}{df_0}\right)^{i-j}
\]
for $i=2,\ldots, d$, and
\[H_i\equiv_Q0\] 
for $i=2,\ldots, d-1$.
By multiplying this equation by $\bigr(df_0\bigl)^i$
	for each $i=2,\ldots,d-1$, we have
	\begin{dmath*}
		0\equiv_Q\sum_{i\geq j\geq0}\binom{d-j}{i-j}F_j(-F_1)^{i-j}\bigr(df_0\bigl)^j,
	\end{dmath*}
i.e.
\[H_{\mathrm{outer};\, i}(\F)\equiv_Q0\]
for $i=1,\ldots,d-1$.
Since $h_{\mathrm{outer};\, i}(\F)\not\in Q$ for $i=1,\ldots,d-1$,
\[H_{\mathrm{outer};\, i}(\F)=0\]
as a polynomial with $O_K$ coefficients for $i=1,\ldots,d-1$.
The remaining proof is omitted as it is the same as the case where \(\mac X_0\).
\end{proof}

\section{Proof of parts $(1)$ and $(2)$ of Theorem \ref{1.1}}

Let $k$ be an algebraically closed field, let $P := [1:0:\cdots:0] \in \mathbb{P}^{n+1}(k)$ be a point, and let $\mathcal{X} \subset \mathbb{P}^{n+1}(k)$ be an irreducible hypersurface of degree $d$, and let $F(X_0,\X):=\sum_{i=0}^dF_iX_0^{d-i}$ be the defining equation of $\mac X$ where $F_i\in k[\X]_i$ for $i=0,\ldots, d$.
Let $(x_0, \ldots, x_n) = (X_0 / X_{n+1}, \ldots, X_n / X_{n+1})$ be a system of affine coordinates, and define 
\[
f(x_0, x_1, \ldots, x_n) := F(x_0, \ldots, x_n, 1).
\]
Then 
\[
f(x_0, x_1, \ldots, x_n) =\sum_{i=0}^df_ix_0^{d-i}
\]
where $f_i\in k[x_1,\ldots,x_{n+1}]$ with ${\rm deg}(f_i)\leq i$ for $i=0,\ldots, d$.
We set 
\[K := k(x_1,\ldots,x_n)\quad{\rm and}\quad L := \mathrm{Quot}\left( k[x_0,x_1,\ldots,x_n] / \big(f(x_0, x_1, \ldots, x_n)\big) \right).
\]
Note that $L = k(\mathcal{X})$ and $K = k(\mathbb{P}^n(k))$.
In this setting, the field extension induced by the point projection $\pi_P \colon \mac X \dashrightarrow \mathbb{P}^n(k)$ corresponds to $L / K$. 
Let $\mathrm{Aut}_K(L)$ denote the group of automorphisms of $L$ that act as the identity on $K$. This group is naturally identified with $G_{\pi_P}$.

We assume that $P$ is an extendable Galois point for $\mathcal{X}$.  
As follows from the proof of Proposition \ref{2.6}, the Galois group $G_{\pi_P}$ is generated by an automorphism $g$ induced by a matrix of the form:
\[
\begin{pmatrix} 
	s_{11} & s_{12} & \dots  & s_{1\,n+2} \\ 
	0 & 1 & \dots  & 0 \\ 
	\vdots & \vdots & \ddots & \vdots \\ 
	0 & 0 & \dots  & 1
\end{pmatrix}
\]
where $s_{11}^{d-1} = 1$ if $\mac X$ is smooth at $P$ and $s_{11}^d= 1$ if $P\not\in \mac X$.
Moreover, if $\mac X$ is smooth at $P$ and $\rmc(k)$ is coprime to $d-1$, or $P\not\in \mac X$ and $\rmc(k)$ is coprime to $d$, then $s_{11}\neq1$.
The element $\sigma \in {\rm Aut}_K(L)$ corresponding to $g$ is expressed as:
\[
\sigma(x_0) = s_{11}x_0 + s_{12}x_1 + s_{13}x_2 + \cdots + s_{1\,n+1}x_n + s_{1\,n+2}.
\]
In other words, $\sigma(x_0)$ can be written as:
\[
\sigma(x_0) = ax_0 + b,
\]
where $a \in k^{\ast}$ and $b \in k[x_1,\ldots,x_n]$ is a polynomial of degree at most one.
\begin{lem}\label{3.1}
	Let $k$ be a field of $\rmc(k)\neq3$, and
	\[f(x, \y) = x^3 + ax +b\]
	be an irreducible polynomial with $a,b \in k[\y]$ where $x, y_1, \ldots, y_n$ are indeterminates.
We set 
\[K := k(\y)\quad {\rm and}\quad L := {\rm Quot}(k[x, \y] / (f(x, \y))).\]	
We assume that there is a non-trivial automorphism $g \in {\rm Aut}_K(L)$ such that 
	\[ 
	g(x) = Ax + B, 
	\] 
	where $A \in k$ with $A\not=1$, $A^3 = 1$, and $B \in k[\y]$.
	Then, we have $a=B=0$.
\end{lem}
\begin{proof}
	Since $f(g(x), \y) =0$,
	we have
	\[
	\begin{aligned}
		0 &= \bigl(Ax + B\bigr)^3 + a\bigl(Ax + B\bigr) +b\\
		&=(x^3 + 3A^2Bx^2 +3AB^2x +B^3)+ (aAx +aB) + b\\
	&=x^3 + 3A^2Bx^2 + \bigl(3AB^2 + aA\bigr)x + B^3 + aB + b.
	\end{aligned}
	\]
	Since $x^3+ax+b=0$,
	we get
	\[
	\begin{aligned}
		0&=3A^2Bx^2+\bigl(3AB^2 + aA-a\bigr)x + B^3 + aB.
	\end{aligned}
	\]
	Since $1, x, x^2$ form a basis of $L$ over $K$, 
	\begin{equation}\label{eq7}
		\begin{cases}
			\begin{split}
				0=&3A^2B,\\
				0=&3AB^2 + aA-a,\\
				0=&B^3 + aB.
			\end{split}
		\end{cases}
	\end{equation}
	Since $A\neq0$ and $\rmc(k)\neq 3$,
	from the first equation in $(\ref{eq7})$, 
	we get $B=0$.
	Since $A\neq 1$,
	from the second equation in $(\ref{eq7})$, 
	we get $a=0$.
\end{proof}
If $\rmc(k)=3$, then for $a\in k$ we have $a^3=1$ if and only if $a=1$.
\begin{lem}\label{3.2}
	Let $k$ be a field of $\rmc(k)=3$, and 
	\[ f(x,\y) = x^3 + ax^2 +bx +c, \]
	be an irreducible polynomial with $a,b,c \in k[\y]$ where $x, y_1, \ldots, y_n$ are indeterminates.
We set 
\[K := k(\y)\quad {\rm and}\quad L := {\rm Quot}(k[x, \y] / (f(x, \y))).\]	
We assume that there is a non-trivial automorphism $g \in {\rm Aut}_K(L)$ such that 
	\[ g(x) =x + B, \] 
	where $B\in k[\y]$ is a non-zero polynomial of degree at most one. 
	Then, we have $a=0$, and $b=-B^2$. 
\end{lem}
\begin{proof}
Since $g$ is a non-trivial automorphism, $B\neq0$.
	Since $f(g(x), \y) = 0$ and $x^3+ax^2+bx+c=0$, we have an equation
	\[\begin{aligned}
		0&=(x+B)^3+a(x+B)^2+b(x+B)+c\\
	&=(x^3+B^3)+a(x^2+2Bx+B^2)+b(x+B)+c\\
		&=2aBx + B^3 + aB^2 + bB.
	\end{aligned}\]
	Since $1, x, x^2$ form a basis of $L$ over $K$, 
	we have
	\begin{equation}\label{eq71}
	\begin{cases}
		\begin{split}
0=&2aB,\\
0=&B^3 + aB^2 + bB.
		\end{split}
	\end{cases}
\end{equation}
Since $\rmc(k)\neq 2$ and $B\neq0$, from the first equation in $(\ref{eq71})$, it follows that $a=0$.
From the second equation in $(\ref{eq71})$, we obtain $B^3+bB^2=0$.
Since $B\neq 0$, we have $b=-B^2$.
\end{proof}
\begin{thm}\label{33}
	Let $k$ be an algebraically closed filed,
	let $P:=[1:0:\cdots:0]\in\mb P^{n+1}(k)$ be a point, 
	and let $\mac X\subset \mb P^{n+1}(k)$ be an irreducible hypersurface of degree $4$.
	We express the defining equation $F(X_0,\ldots, X_{n+1})$ of $\mac X$ as 
	\[F(X_0,\ldots, X_{n+1})=\sum_{i=0}^4F_iX_0^{4-i}\]
	where $F_i\in k[X_1,\ldots, X_{n+1}]_i$ for $i=0,\ldots,4$.
	We assume that $\mac X$ is smooth at $P$. 
	The point $P$ being an extendable inner Galois point for $\mac X$ is equivalent to the following condition holding:
\begin{enumerate}
\item[$(1)$]If $\rmc(k)=3$, then 
\[F_2=0\quad\mathrm{and}\quad F_3=-G^2F_1\]
	where $G\in k[\X]$ is a non-zero form $G\in k[\X]_1$.
	\item[$(2)$]If $\rmc(k)\neq3$, then 
	\[3F_1F_3-F_2^2=0.\]
\end{enumerate}
In both cases, $G_{\pi_P}$ is a cyclic group of order $3$.
\end{thm}
\begin{proof}
	We assume that $\rmc(k)=3$.
	Since $\mac X$ is smooth at $P$, 
	$F_0=0$ and $F_1\neq 0$.
	Let $A\in{\rm GL}(n+1,k)$ be a matrix such that 
\[A^{\ast}F_1=X_{n+1}.\]
We set $B:=\begin{pmatrix}
		1& 0 \\
		0 & A
	\end{pmatrix}\in{\rm GL}(n+2,k)$.
	Then we have
	\[ B^{\ast}F(X_0,\X)=X_{n+1}X_0^3+\sum_{i=2}^4\left(A^{\ast}F_i\right)X_0^{4-i}.\]
	If $P$ is an extendable inner Galois point for $\mac X$, then
$g\in G_{\pi_P}$ is induced by a matrix of the form:
\[
\begin{pmatrix} 
1 & s_{12} & \dots  & s_{1\,n+2} \\ 
	0 & 1 & \dots  & 0 \\ 
	\vdots & \vdots & \ddots & \vdots \\ 
	0 & 0 & \dots  & 1
\end{pmatrix}.
\]
By Lemma \ref{3.2},
	$A^{\ast}F_2=0$ and 
	there exists a form $G'\in k[\X]$ such that
	\[A^{\ast}F_3=-{G'}^2X_{n+1}.\]
Since $A\in{\rm GL}(n+1,k)$, we have $F_0=0$.
	Let 
	\[G:={(A^{-1})}^{\ast}G'.\]
	Then
\[F_3=-G^2F_1.\]

Conversely, we assume that $F_2=0$ and
$F_3=-G^2F_1$
where $G\in k[\X]$ is a non-zero form.
	We set 
	\[G=\sum_{i=1}^{n+1}c_iX_i\] where $c_i\in k$ for $i=1,\ldots, n+1$.
Let $C:=(c_{ij})\in{\rm GL}(n+2,k)$ be the matrix defined by 
\begin{equation*}
c_{ij}=
	\begin{cases}
	c_{j-1}&{\rm if}\ i=1\ {\rm and}\ j\geq 2, \\
		1&{\rm if}\ i=j,\\
		0&{\rm otherwise}.
	\end{cases}
\end{equation*}
Since $F_3=-G^2F_1$, we have
	\begin{dmath*}\label{eq9}
C^{\ast}\F
=\sum_{i=1,3,4}F_i(X_0+G)^{4-i}
= F_1(X_0^3+G^3)+F_3\bigl(X_0+G\bigr)+F_4
=F_1(X_0^3+G^3)+F_3X_0-G^3F_1+F_4
=\F.
\end{dmath*}
Thus, the matrix $C$ defines an automorphism $g_C$ of $\mathcal X$. 
Since $\rmc(k)=3$, the order of $g_C$ is also $3$. 
Moreover, Since the matrix \( C\) acts trivially on the components \( X_1 \) through \( X_{n+1} \), 
we get
\[\pi_P\circ g_C=\pi_P\]
	where $\pi_P\co \mac X \dashrightarrow\mb P^n(k)$ is the projection with center $P$.
Then $|G_{\pi_P}|=3$, and hence $P$ is an extendable inner Galois point for $\mac X$.

We assume that $\rmc(k)\neq 3$.
If the point $P$ is an extendable inner Galois point for $\mac X$, then
by Proposition \ref{2.7}, we have 
\[F_i=F_1G_{i-1}\]
where $G_{i-1}\in k[\X]_{i-1}$ for $i=2,3$.
We set 
\[G_1=\sum_{i=1}^{n+1}a_iX_i\]
	where $a_i\in k$ for $i=1,\ldots, n+1$.
	Let $A:=(a_{ij})\in {\rm GL}(n+2,k)$ be the matrix defined by
	\begin{equation*}
		a_{ij}=
		\begin{cases}
			-\frac{a_{j-1}}{3}&{\rm if}\ i=1\ {\rm and}\ j\geq 2, \\
			1&{\rm if}\ i=j,\\
			0&{\rm otherwise}.
		\end{cases}
	\end{equation*}
Since $F_2=F_1G_1$,
we have
\begin{equation}\label{u1}
		\begin{split}
A^{\ast}F(X_0,\X)=&\sum_{i=1}^4F_i\left(X_0-\frac{1}{3}G_1\right)^{4-i}\\
=&F_1X_0^3+\left(-\frac{1}{3}F_2G_1+F_3\right)X_0\\
&+\frac{2}{27}F_2G_1^2-\frac{1}{3}G_1F_3+F_4.
\end{split}
\end{equation}
Since the point $P$ is fixed under the action of the matrix $A$, $P$ is an extendable inner Galois point for the quartic hypersurface defined by $A^{\ast}F(X_0,\X)=0$.
By Proposition \ref{2.7}, we have
$F_3-\frac{1}{3}F_2G_1=0$.
By multiplying $F_1$, we get
\begin{equation}\label{s1}
3F_1F_3-F_2^2=0.
\end{equation}
	
	Conversely, we assume that the equation $(\ref{s1})$ holds.
Then there exists a form $G_1\in k[\X]_1$ such that
\[F_2=F_1G_1.\]
By using $G_1=\sum_{i=1}^{n+1}a_iX_i$,
	we define the matrix $A$ as before. 
By the equations $(\ref{u1})$ and $(\ref{s1})$, we have
	\[A^{\ast}F(X_0,\X)=X_{n+1}X_0^3+H_4\]
where $H_4\in k[\X]_4$.
	By Proposition \ref{2.6}, it follows that $P$ is an extendable inner Galois point for $\mathcal{X}$.
\end{proof}
Let \( k \) be an algebraically closed field of $\rmc(k)=3$, and $\mathcal X\subset \mathbb P^{n+1}(k)$ be an irreducible quartic hypersurface defined as the zero locus of the polynomial
\[
F(X_0, \ldots, X_{n+1}) = F_1X_0^3 + \sum_{i=3}^4 F_iX_0^{4-i},
\]
where \( F_i \in k[\X]_i \) for \( i = 1, 3, 4 \).
In what follows, we will show that \( P := [1:0:\cdots:0] \in \mathbb{P}^{n+1}(k) \) is an inner Galois point for \( \mathcal{X} \), then \( P \) is, in fact, an extendable inner Galois point for \( \mathcal{X} \).

The following fact is known for Galois extensions of degree \( 3 \).
This result is standard; see [\ref{bio:ma},\ {\rm Theorem}\ 16.8.5].
\begin{thm}\label{3.4}
	Let $K$ be a field with $\c(K)\neq2$ and $f(x) = x^3 + a_1x^2 + a_2x + a_3 \in K[x]$ be a separable irreducible polynomial, 
	let $L$ be the splitting field of $f(x)$ over $K$, and let $G$ be the Galois group of $L/K$.
	Let $\Delta$ be the discriminant of $f(x)$, i.e., 
\[\Delta := \prod_{i < j} (\alpha_i - \alpha_j)^2\] where $\alpha_1, \alpha_2, \alpha_3$ are the roots of $f(x)$.
Define
	\[
	b_1 := a_1a_2 - 3a_3\quad{\rm and}\quad  b_2 := a_2^3 + 9a_3^2 - 6a_1a_2a_3 + a_1^3a_3.
	\]
	Then, we have the following:
\begin{enumerate}
\item[$(1)$]$\Delta = b_1^2 - 4b_2$,
\item[$(2)$]If $\Delta$ is a square in $K$, then [$L:k]=3$, and $G$ is a cyclic group of order $3$.
\item[$(3)$]If $\Delta$ is not a square in $K$, then [$L:k]=6$, and $G$ is the symmetric group $\mac S_3$.
\end{enumerate}	
\end{thm}
If $\rmc(K)=3$ and $a_1 = 0$, then $b_1 = 0$ and $b_2 = a_2^3$. 
In this case, 
\[
\Delta=-a_2^3.
\]
\begin{lem}\label{3.5}
	Let \( k \) be an algebraically closed field of $\rmc(k)=3$, and
\[f(x, \y) = x^3 + a_2 x + a_3 \]
	be an irreducible polynomial such that \( a_2, a_3 \in k[\y] \) with \( a_2 \neq 0 \) and \( \deg(a_2) \leq 3 \), and $x, y_1, \ldots, y_n$ are indeterminates.
	Let $K := k(\y)$ and $L := {\rm Quot}(k[x, \y] / (f(x, \y)))$.  
We	assume that \( L / K \) is a Galois extension.  
	Then, there exists a polynomial \( B\in k[\y] \) such that $a_2 = -B^2$
\end{lem}
\begin{proof}
	Since \( L / K \) is a Galois extension, by Theorem \ref{3.4}, there exists $b \in K$ such that 
	\[-a_2^3 =b^2.\]  
Since $a_2 \in k[\y]$, it follows that 
	\[b \in k[\y].\]  
Since $-a_2^3 =b^2$, we have $3{\rm deg}(a_2)=2{\rm deg}(b)$.
Since \( \deg(a_2) \leq 3 \), we obtain \( \deg(a_2) = 0 \) or \( 2 \).
	If \( \deg(a_2) = 0 \), then \( a_2 \in k \). Since \( k \) is algebraically closed, there exists \( B \in k \) such that \( a_2 = -B^2 \).  
We assume that \( \deg(a_2) = 2 \).  
	Let 
\[a_2 = \prod_{i=1}^2 s_i(\y)\quad {\rm and}\quad b = \prod_{j=1}^3 t_j(\y)\]  
	where \( s_i(\y), t_j(\y) \in k[\y] \) are linear polynomials for \( i = 1, 2 \) and \( j = 1, 2, 3 \).  
We assume \( s_1(\y) \) is not associated to \( s_2(\y) \), i.e. $ts_1(\y)\neq s_2(\y)$ for any $k\in k$.
Since \( k[\y] \) is a unique factorization domain and \( b^2 = -a_2^3 \), we may assume, after reordering indices if necessary, that \( s_1(\y) \) is associated to \( t_1(\y) \) and \( s_2(\y) \) is associated to \( t_j(\y) \) for \( j = 2, 3 \).  
	However, from \( b^2 = -a_2^3 \), it follows that  
	\[
	s_1(\y)^2 s_2(\y)^4 = -u s_1(\y)^3 s_2(\y)^3
	\]
	for some \( u \in k \).  
	This contradicts the fact that \( k[\y] \) is a unique factorization domain.  
	Thus, \( s_1(\y) \) is associated to \( s_2(\y) \).  
	It follows that 
	\[a_2 = v s_1(\y)^2 \] for some \( v \in k \).  
	Since \( k \) is algebraically closed, there exists \( w \in k \) such that \( w^2 = -v \).  
	Let \( B := w s_1(\y) \in k[\y] \). Then \( a_2 = -B^2 \).
\end{proof}

\begin{thm}\label{3.6}
Let $k$ be an algebraically closed filed of $\rmc(k)=3$, 
	let $P:=[1:0:\cdots:0]\in\mb P^{n+1}(k)$ be a point, 
	and let $\mac X\subset \mb P^{n+1}(k)$ be an irreducible hypersurface of degree $4$
	such that $\mac X$ is smooth at $P$.
	We express the defining equation $F(X_0,\ldots, X_{n+1})$ of $\mac X$ as 
	\[F(X_0,\ldots, X_{n+1})=\sum_{i=1}^4F_iX_0^{4-i}\]
	where $F_i\in k[X_1,\ldots, X_{n+1}]_i$ for $i=1,\ldots,4$.
If $P$ is an inner Galois point for $\mac X$ and $F_2=0$,
then $P$ is extendable.
\end{thm}
\begin{proof}
By lemma \ref{3.5}, 
there exist a form $G\in k[\X]$ such that 
\[F_3=-G^2F_1.\]
By part $(1)$ of Theorem \ref{33},
it follows that the point $P$ is an inner extendable Galois point for $\mac X$.  	
\end{proof}

\section{Proof of parts $(3)$ and $(4)$ of Theorem \ref{1.1}}

\begin{lem}\label{4.1}
	Let \( k \) be a field of \(\rmc(k)\neq 2 \), and 
\[ f(x, y) = x^4 + ax^2 + bx + c \]  
be an irreducible polynomial where $a,b,c \in k[\y]$ and $x, y_1, \ldots, y_n$ are indeterminates.
We set 
\[K := k(\y)\quad {\rm and}\quad L := {\rm Quot}(k[x, \y] / (f(x, \y))).\]	
We assume that there is a non-trivial automorphism $g \in {\rm Aut}_K(L)$ such that 
\[g(x) = Ax + B \]  
	where $A \in k$ with $A^2\neq1$, \( A^4 = 1 \), and $B\in k[\y]$.  
Then, we have $B = a= b= 0$.
\end{lem}
\begin{proof}
Since \( f(g(x), \y) = 0 \) and  $x^4 + ax^2 + bx + c=0$ , we have
\[
\begin{aligned}
0=&(Ax+B)^4+a(Ax+B)^2+b(Ax+B)+c\\
=&x^4 + 4A^3Bx^3 + (6A^2B^2 + aA^2)x^2\\ 
&+(4AB^3 + 2aAB + bA)x +B^4 + aB^2 + bB + c\\
=&4A^3Bx^3 + (6A^2B^2 + aA^2-a)x^2\\ 
&+(4AB^3 + 2aAB + bA-b)x +B^4 + aB^2 + bB.
\end{aligned}
\]
Since $1, x, x^2,x^3$ form a basis of $L$ over $K$, 
we obtain the following equations:
\begin{equation}\label{y1}
	\begin{cases}
		\begin{split}
0=&4A^3B,\\
0=&6A^2B^2 + aA^2-a, \\
0=&4AB^3 + 2aAB + bA-b, \\
0=&B^4 + aB^2 + bB.
\end{split}
\end{cases}
\end{equation}
From the first equation in $(\ref{y1})$, since \( A \neq 0 \), we have $B = 0$.  
By substituting \( B = 0 \) into the second equation in $(\ref{y1})$, we get $a(A^2-1)=0$.  
	Since \( A^2 \neq 1 \), it follows that \( a= 0 \).  
Similarly, by substituting \( B = 0 \) into the third equation in $(\ref{y1})$,
we get $b(A-1)$.  
Since $A \neq 1$, we conclude that $b= 0$.
\end{proof}
If $\rmc(k)=2$, then for $a\in k$ we have $a^4=1$ if and only if $a=1$.
\begin{lem}\label{4.2}
	Let \( k \) be a field of \( \rmc(k)=2 \), and  
\[f(x, \y) = x^4 + a x^3 + b x^2 + c x + d \]  
be an irreducible polynomial where $a, b, c, d \in k[\y]$ and $x, y_1, \ldots, y_n$ are indeterminates.  
We set 
\[K := k(\y)\quad {\rm and}\quad L := {\rm Quot}(k[x, \y] / (f(x, \y))).\]	
We assume that $f(x,\y)$ is separable, and there is a non-trivial automorphism $g \in {\rm Aut}_K(L)$ such that 
\[ g(x) =x + B\]
 where $B \in k[\y]$ is a polynomial of degree at most \( 1 \).  
	Then, the order of $g$ is two, and we have $a=0$ and $B^3 + b B + c = 0$.
\end{lem}
\begin{proof}
Since $g$ is a non-trivial automorphism, we obtain $B\neq0$.
Since $\rmc(k)=2$, we have $g\circ g(x)=x$.
Thus, the order of $g$ is two.
Since \( f(g(x), \y) = 0 \) and $x^4 + a x^3 + b x^2 + c x + d=0$, we have
\begin{dmath*}
0=(x + B)^4 + a (x+B)^3 + b (x+ B)^2 + c (x + B) + d
=(x^4 + B^4) + a (x^3 + B x^2 + B^2 x + B^3) + b (x^2 + B^2) + c (x + B) + d
=x^4 + ax^3 + (aB + b) x^2 + (aB^2 + c) x + B^4 + a B^3 + b B^2 + c B + d
=aBx^2 +aB^2x +B^4 + a B^3 + b B^2 + c B.
\end{dmath*}
Since $1, x, x^2,x^3$ form a basis of $L$ over $K$, 
we obtain the following equations:
\begin{equation}\label{y2}
	\begin{cases}
		\begin{split}
0&=aB,\\
0&=aB^2, \\
0&=B^4 + a B^3 + b B^2 + c B.
		\end{split}
\end{cases}
\end{equation}
Since $B\neq0$, from the first equation in $(\ref{y2})$, we have $a=0$.
From the third equation in $(\ref{y2})$, we obtain $B^4 + b B^2 + cB = 0$.
Since $B\neq 0$, we get $B^3+ b B+ c= 0$. 
\end{proof}
\begin{thm}\label{43}
	Let $k$ be an algebraically closed filed, 
	let $P:=[1:0:\cdots:0]\in\mb P^{n+1}(k)$ be a point, 
	and let $X\subset \mb P^{n+1}(k)$ be an irreducible hypersurface of degree $4$.
We express the defining equation $F(X_0,\ldots, X_{n+1})$ of $X$ as 
\[F(X_0,\ldots, X_{n+1})=X_0^4+\sum_{i=1}^4F_iX_0^{4-i}\]
	where $F_i\in k[X_1,\ldots, X_{n+1}]_i$ for $i=1,\ldots,4$.
We assume that $P\not\in \mac X$. 
The point $P$ being an extendable outer Galois point for $\mac X$ is equivalent to the following condition holding:
\begin{enumerate}
\item[$(1)$]If $\rmc(k)=2$, then 
\[F_1=0,\]
and for the polynomial $T^3 + F_2T + F_3$ where $T$ is an indeterminate,
there exist forms \( B_i \in k[\X] \) with $B_i\neq0$ for $i=1,2,3$ such that 
\begin{align*}
	\begin{split}
		T^3 + F_2T +F_3=(T-B_1)(T-B_2)(T-B_3).
	\end{split}
\end{align*}
In this case, the Galois group of $k(\mathcal{X})/k(\mathbb{P}^n(k))$ is the direct product of two cyclic groups of order 2.
\item[$(2)$]If $\rmc(k)\neq2$, then 
\[3F_1^2-8F_2=0\quad\mathrm{and}\quad F_1^3-16F_3=0.\]
In this case,  the Galois group of $k(\mathcal{X})/k(\mathbb{P}^n(k))$ is a cyclic group of order $4$.
\end{enumerate}
\end{thm}
\begin{proof}
Since $P\not\in\mac X$, $F_0\neq0$.
For simplicity, we may assume that $F_0=1$.
We assume that $\rmc(k)=2$.
If $P$ is an extendable outer Galois point for $\mac X$, then
$g\in G_{\pi_P}$ is induced by a matrix of the form:
\[
\begin{pmatrix} 
	1 & s_{12} & \dots  & s_{1\,n+2} \\ 
	0 & 1 & \dots  & 0 \\ 
	\vdots & \vdots & \ddots & \vdots \\ 
	0 & 0 & \dots  & 1
\end{pmatrix}.
\]
We set $B:=\sum_{i=1}^{n+1}s_{1\,i+1}X_i$.
By Lemma \ref{4.2}, we have
$F_1=0$ and 
$B^3+F_2B+F_3=0$.
Since $[k(\mac X):k(\mb P^n(k))]=4$, by Lemma \ref{4.2}, it follows that $G_{\pi_P}$ is the direct product of two cyclic groups of order $2$.
Then, there exist non-zero forms $B_i$
$ \in k[\X]$ for $i=1,2,3$ such that 
\begin{align*}
	\begin{split}
		T^3 + F_2T +F_3=(T-B_1)(T-B_2)(T-B_3)
	\end{split}
\end{align*}
where $T$ is an indeterminate.

Conversely, we assume that 
\[F_1=0\] and 
there exist non-zero forms \( B_i \in k[\X] \) for $i=1,2,3$ such that 
\begin{align*}
	\begin{split}
		T^3 + F_2T +F_3=(T-B_1)(T-B_2)(T-B_3)
\end{split}
\end{align*}
where $T$ is an indeterminate.
Then 
\[B_i^3+F_2B_i+F_3=0\]
for $i=1,2,3$.
It follows that
\begin{dmath}\label{eq10}
\sum_{j=0,2,3,4}F_j\bigl(X_0+B_i\bigr)^{4-j}=\bigl(X_0^4+B_i^4\bigr)+F_2\bigl(X_0^2+B_i^2\bigr)+F_3\bigl(X_0+B_i\bigr)+F_4(\X)=\sum_{j=0,2,3,4}F_jX_0^{4-j}+B^4_i+F_2B^2_i+F_3B
	=\sum_{j=0,2,3,4}F_jX_0^{4-j}
\end{dmath}
for $i=1,2,3$.
	We set $B_i=\sum_{j=1}^{n+1}a_{ij}X_j$ where $a_{ij}\in k$ for $i=1,2,3$ and $j=1,\ldots, n+1$.
Let $A_i:=(a(i)_{s\,t})\in{\rm GL}(n+2,k)$ be a matrix such that 
\begin{equation*}
	a(i)_{s\,t}=
	\begin{cases}
		a_{i\,t-1}&{\rm if}\ s=1\ {\rm and}\ t\geq 2, \\
		1&{\rm if}\ s=t,\\
		0&{\rm otherwise}.
	\end{cases}
\end{equation*}
for $i=1,2,3$.
By the equation $(\ref{eq10})$,
we get that
the matrix \( A_i \) defines an automorphism $g_{A_i}$ of \( X \) for $i=1,2,3$. 
Since $\rmc(k)=2$, the order of $g_{A_i}$ is also $2$. 
Moreover, Since the matrix \( A_i \) acts trivially on the components \( X_1 \) through \( X_{n+1} \), 
\[\pi_P\circ g_{A_i}=\pi_P\]
where $\pi_P\co \mac X\dashrightarrow\mb P^n(k)$ is the projection with center $P$.
Then $|G_{\pi_4}|=4$, and hence $P$ is an extendable inner Galois point for $\mac X$.
Since the order of $g_{A_i}$ is $2$ for $i=1,2,3$, we have $G_{\pi_P}$ is the direct product of two cyclic groups of order $2$.

We assume that $\rmc(k)\neq 2$.
We assume that $P$ is an extendable outer Galois point for $\mac X$.
We set $F_1=\sum_{i=1}^{n+1}a_iX_i$.
Let $A:=(a_{ij})\in {\rm GL}(n+2,k)$ be the matrix defined by
\begin{equation*}
	a_{ij}=
	\begin{cases}
		-\frac{a_{j-1}}{4}&{\rm if}\ i=1\ {\rm and}\ j\geq 2, \\
		1&{\rm if}\ i=j,\\
		0&{\rm otherwise}.
	\end{cases}
\end{equation*}
Then, we have 
\begin{equation}\label{i1}
	\begin{split}
	A^{\ast}F(X_0,\X)=&\left(X_0-\frac{F_1}{4}\right)^4+F_1\left(X_0-\frac{F_1}{4}\right)^3+F_2\left(X_0-\frac{F_1}{4}\right)^2\\
&+F_3\left(X_0-\frac{F_1}{4}\right)+F_4\\
=&X_0^4+\left(-\frac{3}{8}F_1^2+F_2\right)X_0^2+\left(\frac{1}{8}F_1^3-\frac{1}{2}F_1F_2+F_3\right)X_0\\
&-\frac{3}{256}F_1^4+\frac{1}{16}F_1^2F_2-\frac{1}{4}F_1F_3+F_4.
\end{split}
\end{equation}
By Proposition \ref{2.7}, we have
\begin{equation}\label{a1}
3F_1^2-8F_2=0
\end{equation}
and
\begin{align}\label{a2}
	\begin{split}
	F_1^3-4F_1F_2+8F_3=0.
	\end{split}
\end{align}
By substituting  the equation $(\ref{a1})$ into the equation $(\ref{a2})$, the following equation $(\ref{a3})$ holds:
\begin{equation}\label{a3}
F_1^3-16F_3=0.
\end{equation}

Conversely, we assume that the equations $(\ref{a1})$ and $(\ref{a3})$ hold.
From the equations $(\ref{a1})$ and $(\ref{a3})$, we have the equation $(\ref{a2})$.
By using $F_1=\sum_{i=1}^{n+1}a_iX_i$,
we define the matrix $A$ as before. 
By the equations $(\ref{i1})$, $(\ref{a1})$, and $(\ref{a2})$, we have 
\[A^{\ast}F(X_0,\X)=X_0^4+H_4\]
where $H_4\in k[\X]_4$.
By Proposition \ref{2.6}, it follows that $P$ is an extendable outer Galois point for $\mathcal{X}$.	
\end{proof}
Let \( k \) be an algebraically closed field of \( \mathrm{char}(k) = 2 \), and
\( \mathcal{X} \subset \mathbb{P}^{n+1}(k) \) be an irreducible  quartic hypersurface 
defined as the zero locus of the polynomial
\[
F(X_0, \ldots, X_{n+1}) =X_0^4 + \sum_{i=2}^4 F_iX_0^{4-i},
\]
where \( F_i\in k[\X]_i \) for \( i = 2, 3, 4 \). 
In what follows, we show that if \( P := [1:0:\cdots:0] \in \mathbb{P}^{n+1}(k) \) is an outer Galois point for \( \mathcal{X} \)
such that $G_{\pi_P}$ is the direct product of two cyclic groups of order $2$, then \( P \) is an extendable outer Galois point for \( \mathcal{X} \).

The following fact is known for Galois extensions of degree \( 4 \).
This result is standard; see [\ref{bio:ma},\ {\rm Proposition}\ 16.9.8].
\begin{thm}\label{4.4}
	Let $K$ be a field, and let $f(x) = x^4 + a_1x^3 + a_2x^2 + a_3x + a_4 \in K[x]$ be a separable irreducible polynomial. 
	Let $L$ be the splitting field of $f(x)$ over $K$, and $G$ be the Galois group of $L/K$.
	Define
	\[
	b_1 := -a_2, \quad b_2 := a_1a_3 - 4a_4, \quad b_3 := -a_4(a_1^2 - 4a_2) - a_3^2.
	\]
	Let $g(t)$ be the resolvent cubic of $f$.
	Then, we have the following:
	\begin{enumerate}
		\item[$(1)$]$g(t)=t^3 + b_1t^2 + b_2t + b_3$,
		\item[$(2)$]If $g(t)= (t - a)(t - b)(t - c)\ \text{where } a, b, c \in K$, then $G$ is the direct product of two cyclic groups of order $2$,
		\item[$(3)$]If $g(t)$ has only one root in $K$, then $G$ is a dihedral group $D_4$ or a cyclic group of order $4$,
		\item[$(4)$]If $g(t)$ is irreducible over $K$, then $G$ is a a symmetric group $\mac S_4$ or an alternating group $A_4$.
	\end{enumerate}
\end{thm}
If $\rmc(K)=2$ and $a_1 = 0$, then the coefficients simplify as follows:
\[
b_1 = a_2, \quad b_2 = 0, \quad b_3 = a_3^2.
\]
In this case, the cubic polynomial $t^3 + b_1t^2 + b_2t + b_3$ simplifies to 
\[
t^3 + a_2t^2 + a_3^2.
\]
Furthermore, by the change of variables $t = t' + a_2$, the polynomial transforms into 
\[
{t'}^3 + a_2^2t' + a_3^2.
\]
\begin{lem}\label{4.5}
Let \( k \) be an algebraically closed field of $\rmc(k)=2$, and 
\[
f(x, y) = x^4 + a_2 x^2 + a_3 x + a_4
\]	
be an irreducible polynomial
	where \( a_2, a_3, a_4 \in k[\y] \), \( a_3 a_4 \neq 0 \), and \( \deg(a_i) \leq i \) for \( i = 2, 3, 4 \), and $x, y_1, \ldots, y_n$ are indeterminates. 
We set
\[K := k(\y)\quad {\rm and}\quad  L := {\rm Quot}(k[x, \y]/(f(x, \y))).\]
We assume that \( L/K \) is a Galois extension with Galois group is the direct product of two cyclic groups of order $2$.
Then for the polynomial $t^3 + a_2 t + a_3$ where \( t \) is an indeterminate,
there exist non-zero elements \( B_1,B_2,B_3 \in k[\y] \) such that 
\[t^3 + a_2 t + a_3=(t-B_1)(t-B_2)(t-B_3).\]
\end{lem}
\begin{proof}
Since \( L/K \) is a Galois extension, there exist \( b_1, b_2, b_3 \in k(\y) \) such that
\begin{equation}\label{c1}
t^3 + a_2^2 t + a_3^2 = (t - b_1)(t - b_2)(t - b_3).
\end{equation}
Then we have
\[b_i^3 + a_2^2 b_i + a_3^2 = 0\]
 for \( i = 1, 2, 3 \). Since \( a_2, a_3 \in k[\y] \),
it follows that \( b_i \in k[\y] \) for \( i = 1, 2, 3 \). 
Since  $b_i^3 + a_2^2 b_i + a_3^2 = 0$ and $\rmc(k)=2$, we obtain
\[ b_i=\Bigr(\frac{a_3}{a_2+b_i}\Bigl)^2 \]
 for $i=1,2,3$.
Since $b_i\in k[\y]$, it follows that $\frac{a_3}{a_2+b_i}\in k[\y]$ for $i=1,2,3$.
Thus, there exists $B_i\in k[\y]$ such that 
\[b_i=B_i^2\]
 for $i=1,2,3$.
By te equation $(\ref{c1})$, we have $t^3 + a_2^2 t + a_3^2 = (t - B_1^2)(t - B_2^2)(t - B_3^2)$.
By the relationship between the roots and coefficients of the polynomial,
we have equations:
\begin{equation*}
	\begin{cases}
		\begin{split}
			0&=B_1^2+B_2^2+B_3^2,\\
			a_2^2&=B_1^2B_2^3+B_2^2B_3^2+B_1^2B_3^2, \\
			a_3^2&=B_1^2B_2^2B_3^2.
		\end{split}
	\end{cases}
\end{equation*}
Since $\rmc(k)=2$, the above equations are as follows:
\begin{equation*}
	\begin{cases}
		\begin{split}
			0&=B_1+B_2+B_3,\\
			a_2&=B_1B_2+B_2B_3+B_1B_3, \\
			a_3&=B_1B_2B_3.
		\end{split}
	\end{cases}
\end{equation*}
Therefore, we get
 \[t^3 + a_2 t + a_3 = (t - B_1)(t - B_2)(t - B_3).\]
\end{proof}
\begin{thm}\label{4.6}
	Let $k$ be an algebraically closed filed of $\rmc(k)=2$, 
	let $P:=[1:0:\cdots:0]\in\mb P^{n+1}(k)$ be a point, 
	and let $\mac X\subset \mb P^{n+1}(k)$ be an irreducible quartic hypersurface.
	We express the defining equation $F(X_0,\ldots, X_{n+1})$ of $X$ as 
\[F(X_0,\ldots, X_{n+1})=\sum_{i=0}^4F_iX_0^{4-i}\]
	where $F_i\in k[X_1,\ldots, X_{n+1}]_i$ for $i=1,\ldots,4$.
	If $P$ is an outer Galois point for $\mac X$ such that $G_{\pi_P}$ is the direct product of two cyclic groups of order $2$,
	and $F_1=0$,
	then $P$ is extendable.
\end{thm}
\begin{proof}
By lemma \ref{4.5}, we have 
	there exist forms $B_1,B_2,B_3\in k[\X]$ such that
\[T^3 + F_2(\X)T + F_3(\X)=(T-B_1)(T-B_2)(T-B_3)\]
where $T$ is an indeterminate.
	By part $(1)$ of Theorem \ref{43},
	it follows that the point $P$ is an extendable outer Galois point for $\mac X$.  	
\end{proof}


\begin{thebibliography}{99}
\bibitem{ha}\label{bio:acgh}
E. Arbarello, M. Cornalba, P.A. Griffiths, J. Harris, Geometry of algebraic curves, vol. I. Grundlehren der Mathematischen Wissenschaften 267. Springer, New York (1985).
\bibitem{}\label{bio:ma} 
M. Artin, Algebra, 2nd ed., Pearson, Boston, MA, 2010. ISBN: 978$-$0132413770.
%
%
%
%
\bibitem{ha}\label{bio:f06}
S. Fukasawa, Galois points on quartic curves in characteristic 3, Nihonkai Math. J. 17 (2006), no.2, 103$-$110.
\bibitem{ha}\label{bio:f07}
S. Fukasawa, On the number of Galois points for a plane curve in positive characteristic. II, Geom. Dedicata 127 (2007), 131$-$137.
\bibitem{}\label{bio:f08}
S. Fukasawa, On the number of Galois points for a plane curve in positive characteristic, Comm. Algebra {\bf 36} (2008), no.1, 29$-$36.
\bibitem{ha}\label{bio:f14}
S. Fukasawa, Automorphism groups of smooth plane curves with many Galois points, Nihonkai Math. J. 25 (2014), no.1, 69$-$75.
\bibitem{ha}\label{bio:ft14}
S. Fukasawa and T. Takahashi, Galois points for a normal hypersurface, Trans. Amer. Math. Soc., 366 (2014), 1639$-$1658.
\bibitem{ha}\label{bio:f142}
S. Fukasawa, Galois points for a plane curve in characteristic two, J. Pure Appl. Algebra 218 (2014), no.2, 343$-$353.
\bibitem{ha}\label{bio:fmt19}
S. Fukasawa, K. Miura, and T. Takahashi, Quasi-Galois points, I: Automorphism groups of plane curves, Tohoku Math. J. (2) Volume 71, Number 4 (2019), 487$-$494.	
\bibitem{ha}\label{bio:f22}
S. Fukasawa, Automorphism group, Galois points and lines of the generalized Artin-Schreier-Mumford curve, Geom. Dedicata 216 (2022), no.2, Paper No. 19, 9 pp.		
\bibitem{ha}\label{bio:f22gc}
S. Fukasawa, Algebraic-geometric codes with many automorphisms arising from Galois points, preprint arXiv:2211.16823 (2022).
\bibitem{ha}\label{bio:fm23}
S. Fukasawa and T. Miezaki, Galois points for a finite graph, arXiv preprint arXiv:2308.05293 (2023).
\bibitem{ha}\label{bio:hmo18}		
T. Harui, K. Miura and A. Ohbuchi, Automorphism group of plane curve computed by Galois points, II, Proc. Japan Acad. Ser. A Math. Sci. 94 (2018), no.6, 59$-$63.
\bibitem{ha}\label{bio:hmo22}		
T. Harui, K. Miura and A. Ohbuchi, Smooth plane curves with outer Galois points whose reduced automorphism group is $A_5$, Proc. Japan Acad. Ser. A Math. Sci. 98 (2022), no.8, 67$-$71.
\bibitem{ha}\label{bio:rh}
R. Hartshorne, Algebraic geometry, Graduate Texts in Mathematics, No. 52. Springer-Verlag, New York, Heidelberg, 1977.
\bibitem{ha}\label{bio:th21l}
T. Hayashi, Linear automorphisms of smooth hypersurfaces giving Galois points, Bull. Korean Math. Soc. 58 (2021), No. 3, pp. 617$-$635.
\bibitem{ha}\label{bio:th21d}
T. Hayashi, Dihedral group and smooth plane curves with many quasi-Galois points, Bull. Malays. Math. Sci, 44, 4251$-$4267 (2021).
\bibitem{ha}\label{bio:th23g}
T. Hayashi, Galois covers of the projective line by smooth plane curves of large degree, Beitr Algebra Geom 64, 311-365 (2023). 
\bibitem{ha}\label{bio:th23r}
T. Hayashi, Ramification locus of projections of smooth hypersurfaces to be Galois rational covers, Geometriae Dedicata 217.4 (2023)$\co$73.
\bibitem{ha}\label{bio:th24a}
T. Hayashi, Abelian automorphism groups of smooth hypersurfaces with smooth quotient, to appear in Beitr Algebra Geom. 
\bibitem{ha}\label{bio:h06}
M. Homma, Galois points for a Hermitian curve, Comm. Algebra 34 (2006), 4503-4511.
\bibitem{ha}\label{bio:it24}		
A. Ikeda and T. Takahashi, Simultaneous Galois points for a reducible plane curve consisting of nonsingular components, Kodai Math. J. 7 (2024), no.2, 251$-$265.
\bibitem{ha}\label{bio:kty01}
M. Kanazawa, T. Takahashi and H. Yoshihara, The group generated by automorphisms belonging to Galois points of the quartic surface, Nihonkai Math. J., 12 (2001), 89$-$99.
\bibitem{ha}\label{bio:mm63}
H. Matsumura and P. Monsky, On the automorphisms of hypersurfaces, J. Math. Kyoto Univ. 3 (1963/1964), 347$-$361.
\bibitem{ha}\label{bio:mo15}
K. Miura and A. Ohbuchi, Automorphism group of plane curve computed by Galois points, Beitr. Algebra Geom. 56 (2015), no.2, 695$-$702.
\bibitem{ha}\label{bio:my00}
K. Miura and H. Yoshihara, Field theory for function fields of plane quartic curves, J. Algebra, 226 (2000), 283$-$294.
\bibitem{ha}\label{bio:y01f}
H. Yoshihara, Function field theory of plane curves by dual curves, J. Algebra 239 (2001), 340$-$355.
\bibitem{ha}\label{bio:y01g}
H. Yoshihara, Galois points on quartic surfaces, J. Math. Soc. Japan, 53 (2001), 731$-$743.
\bibitem{ha}\label{bio:y03}
H. Yoshihara, Galois points for smooth hypersurfaces, J. Algebra, 264 (2003), 520-534.
\end{thebibliography}
\end{document}